\documentclass[11pt]{article}
\pdfoutput=1
\usepackage{amsmath, amsfonts, amsthm, amssymb, hyperref, setspace, geometry, mathrsfs, multirow}
\title{\textbf{New characterizations for supersolvability of fusion system and $p$-nilpotency of finite groups}\thanks{ \scriptsize\emph{E-mail addresses:}
      zsmcau@cau.edu.cn\,(S. Zhang); zhencai688@sina.com\,(Z. Shen).}}
 
\author{Shengmin Zhang,   Zhencai Shen\\
\quad
\\{\small College of Science,
China Agricultural University,
Beijing 100083, China}}

\date{}
\newtheorem{theorem}{Theorem}[section]

\newtheorem{lemma}[theorem]{Lemma}
\newtheorem{corollary}[theorem]{Corollary}
\theoremstyle{definition}
\newtheorem{definition}[theorem]{Definition}

\newtheorem{remark}[theorem]{Remark}
\newtheorem{example}[theorem]{Example}
\allowdisplaybreaks

\expandafter\let\expandafter\oldproof\csname\string\proof\endcsname
\let\oldendproof\endproof
\renewenvironment{proof}[1][\proofname]{%
  \oldproof[\bfseries\scshape #1]%
}{\oldendproof}

\usepackage{stackengine,scalerel}
\stackMath
\def\trianglelefteqslant{\ThisStyle{\mathrel{%
  \stackinset{r}{.75pt+.15\LMpt}{t}{.1\LMpt}{\rule{.3pt}{1.1\LMex+.2ex}}{\SavedStyle\leqslant}%
}}}
\renewcommand{\unlhd}{\trianglelefteqslant}

\renewcommand{\leq}{\leqslant}
\renewcommand{\geq}{\geqslant}
\begin{document}
\maketitle
\begin{abstract}
Let $p$ be a prime, $S$ be a $p$-group and $\mathcal{F}$ be a saturated fusion system over $S$. Then $\mathcal{F}$ is said to be supersolvable, if there exists a series of $S$, namely $1 = S_0 \leq S_1 \leq \cdots \leq S_n = S$, such that $S_{i+1}/S_i$ is cyclic, $i=0,1,\cdots, n-1$, $S_i$ is strongly $\mathcal{F}$-closed, $i=0,1,\cdots,n$. In this paper, we investigate the characterizations for supersolvability of $\mathcal{F}_S (G)$ under the assumption that certain subgroups of $G$ satisfy different kinds of generalized normalities in section \ref{1003}. In section \ref{Section 4}, we obtain the   result of characterizations for generalized saturated fusion system $\mathcal{F}$ based on the work of F. Aseeri {\it et al} in {{\cite{FJ}}}. Finally, we apply the results in section \ref{1003} and \ref{Section 4} and give characterizations for $p$-nilpotency of finite groups under the assumption that some subgroups of $G$ satisfy different kinds of generalized normalities.

\end{abstract}

\section{Introduction}
All groups considered in this paper will be finite and $G$ always denotes a finite group. Let $S$ be a Sylow $p$-subgroup of $G$, where $p$ is a prime divisor of $|G|$. Then the fusion system of $G$ over $S$, namely $\mathcal{F}_S (G)$, is a fusion category over $S$ which is defined as follows:
\begin{itemize}
\item[(1)] The object of $\mathcal{F}_S (G)$ is the set of all subgroups of $S$.
\item[(2)] For any $P,Q \leq S$, ${\rm Mor}_{\mathcal{F}_S (G)} (P,Q)=\lbrace \phi\,|\,\phi:P\rightarrow Q,\,p \mapsto p^g,\,P^g \leq Q,\,g \in G \rbrace  $.
\end{itemize}
One can  find that $\mathcal{F}_S (G)$ is exactly a saturated fusion system over $S$ by {{\cite[Theorem 2.3]{AK}}}. As is known to all, the structure of $\mathcal{F}_S (G)$ has a strong relationship with the structure of $G$. Hence some structures of finite groups can be generalized into the fusion system $\mathcal{F}_S (G)$. Recall that $G$ is said to be supersolvable, if there exists a chief series, namely
$$1 = N_0 \leq N_1 \leq \cdots \leq N_t = G,$$
such that $N_{i+1}/N_i$ is cyclic, $i=0,1,\cdots,t-1$. As a natural way of generalization, we want to define a similar structure in $\mathcal{F}$, where $\mathcal{F}$ is a fusion system over a $p$-group $S$. Since the objects of $\mathcal{F}$ are exactly the subgroups of $S$, we may restrict the chief series of $G$ into a series of $S$. Note that the normality of subgroups $H$ of $S$ in $G$ represents the invariance of $H$ under the morphisms induced by conjugation of $G$, we may change the normality of $H$ in $G$ into invariance of $H$ under the morphisms in $\mathcal{F}$. Then one can easily find that the invariance of $H$ under the morphisms in $\mathcal{F}$ exactly suits the concept of weakly $\mathcal{F}$-closed property, hence we give the following definition which was introduced by N. Su in {{\cite{SN}}}.
\begin{definition}
Let $\mathcal{F}$ be a saturated fusion system over a $p$-group $S$. Then $\mathcal{F}$ is called supersolvable, if there exists a series of subgroups of $S$, namely:
$$1 = S_0 \leq S_1 \leq \cdots \leq S_n =S,$$
such that $S_i$ is strongly $\mathcal{F}$-closed, $i=0,1,\cdots, n$, and $S_{i+1}/S_i$ is cyclic for any $i=0,1,\cdots,n-1$.
\end{definition}

Now we would like to introduce some concepts which are useful for us to discover the structure of $\mathcal{F}_S (G)$. Let $S$ be a $p$-group and $P$ be a subgroup of $S$. Suppose that $\mathcal{F}$ is a fusion system over $S$. Then $P$ is called $\mathcal{F}$-centric, if $C_S (Q) = Z(Q)$ for all $Q \in P^{\mathcal{F}}$, where $P^{\mathcal{F}}$ denotes the set of all subgroups of $S$ which are $\mathcal{F}$-conjugate to $P$. $P$ is said to be fully normalized in $\mathcal{F}$, if $|N_S (P)| \geq |N_S(Q)|$ for all $Q \in P^{\mathcal{F}}$. $P$ is said to be $\mathcal{F}$-essential, if $P$ is $\mathcal{F}$-centric and fully normalized in $\mathcal{F}$, and ${\rm Out}_{\mathcal{F}} (P)$ contains a strongly $p$-embedded subgroup (see {{\cite[Definition A.6]{AK}}}). Now we are ready to introduce the following concept.
\begin{definition}
Let $p$ be a prime, $\mathcal{F}$ be a saturated fusion system on a finite group $S$. Let 
$$\mathcal{E}_{\mathcal{F}} ^{*}:=\lbrace Q \leq S\,|\,Q ~\text{is $\mathcal{F}$-essential, or $Q =S$}\rbrace. $$
\end{definition}

Now we introduce the section \ref{1003} of our paper. Let $G$ be a finite group and $S$ be a Sylow $p$-subgroup of $G$, where $p\,|\,|G|$. Firstly, we would like to introduce the work of F. Aseeri and J. Kaspczyk in {{\cite{FJ}}}. At the beginning of their passage, they introduced the characterizations for $p$-nilpotency of $G$ under the hypothesis that certain subgroups of $S$ are weakly closed in $S$ with respect to $G$ which are obtained by M. Assad {{\cite{AS}}}. An interesting fact in {{\cite[Theorem 4.3]{LY}}}, S. Liu and H. Yu showed that a subgroup $H$ of $S$ is weakly closed in $S$ with respect to $G$ if and only if $H$ is weakly pronormal in $G$. Note that $H$ is weakly closed in $S$ with respect to $G$ if and only if $H$ is weakly $\mathcal{F}_S (G)$-closed, then the characterizations for $p$-nilpotency of $G$ can be changed into the characterizations for supersolvability of $\mathcal{F}_S (G)$ since if $G$ is $p$-nilpotent, then $\mathcal{F}_S (G) = \mathcal{F}_S (S)$ is supersolvable. As a consequence, they obtained the following result:
\begin{theorem}[{{\cite[Theorem A]{FJ}}}]
Let $p$ be a prime, $G$ be a finite group, and $S$ be a Sylow $p$-subgroup of $G$. Suppose that there exists a subgroup $D$ of $S$ with $1<|D|<|S|$ such that every subgroup of $S$ of order $|D|$ or $p|D|$ is abelian and weakly pronormal in $G$, then $\mathcal{F}_S (G)$ is supersolvable.
\end{theorem}

Then a natural thought of improvement comes to our mind. Does the generalised normality “weakly pronormal” essential? Can we change the weakly pronormal property into other kinds of generalized normalities like weakly $S\Phi$-supplemented property {{\cite[Definition 1.1]{WU}}} and so on? Thus, we analysis the proof of {{\cite[Theorem A]{FJ}}}, and find out the reason why weakly pronormal property has made this theorem holds. In fact, this property is not essential and we claim that $\mathcal{F}_S (G)$ is supersolvable for a bunch of generalised normalities which satisfy certain conditions as we will show in section \ref{1003}. Therefore, in section \ref{1003}, we investigate the sufficient conditions that make $\mathcal{F}_S (G)$  supersolvable, and show two ways to axiomize the sufficient conditions. Then, we give 16 different theorems about characterizations for supersolvability of $\mathcal{F}_S (G)$ based on 16 different kinds of generalised normalities as direct or indirect applications of our axiomatization. All of the definitions of the generalised normalities are given the first time they appear in this paper. Here we would like to define two generalised normalities, which are not researched before.
\begin{definition}\label{1.4}
Let $G$ be a finite group, and $H$ be a subgroup of $G$.
\begin{itemize}
\item[(1)] $H$ is said to be a strong $ICC$-subgroup of $G$, if $H$ is an $ICC$-subgroup {{\cite{GL3}}} for any $H \leq K \leq G$.
\item[(2)] $H$ is said to be strongly $\mathfrak{U}$-embedded in $G$, if $H$ is $\mathfrak{U}$-embedded {{\cite{WS}}} in $K $ for any $H \leq K \leq G$.
\end{itemize}
\end{definition}

However, it seems difficult to migrate other kinds of generalised normalities into the general or exotic saturated fusion system, since the weakly pronormal property has a close relationship with weakly $\mathcal{F}$-closed property in fusion system. More precisely, a subgroup $H$ of $S$ is weakly $\mathcal{F}$-closed if and only if $H$ is weakly pronormal in $G$ if $\mathcal{F} = \mathcal{F}_S (G)$ for some finite group $G$ such that $S \in {\rm Syl}_p (G)$. Therefore, to actually generalise or to cover the work of F. Aseeri and J. Kaspczyk, we need to weaken the condition of weakly $\mathcal{F}$-closed property, where $\mathcal{F}$ is an arbitrary saturated fusion system. Now let $S$ be a $p$-group, and $\mathcal{F}$ be a saturated fusion system over $S$. Let $H$ be a subgroup of $S$. For any $H \leq K \leq S$ and $\alpha \in {\rm Aut}_{\mathcal{F}} (K)$, it follows directly that $\alpha^{*} : H \rightarrow H^{\alpha}, h \mapsto h^{\alpha}$ lies in ${\rm Iso}_{\mathcal{F}} (H,H^{\alpha})$. However, for any $\beta \in {\rm Iso}_{\mathcal{F}} (H,T)$, there does not necessarily exist a subgroup $K \leq H,T$ and an extension $\overline{\beta} \in {\rm Aut}_{\mathcal{F}} (K)$. The argument above inspires us to consider a generalization of weakly $\mathcal{F}$-closed property, and we list it as follows.
\begin{definition}
let $\mathcal{F}$ be a fusion system over a $p$-group $S$, and $H$ be a subgroup of $S$. Then $H$ is said to be semi-invariant in $\mathcal{F}$, if $H$ is $\mathcal{F}$-invariant in any subgroup $K$ of $S$ that containing $H$. Here, $H$ is called $\mathcal{F}$-invariant in $K$, if $H$ is invariant under the action induced by ${\rm Aut}_{\mathcal{F}} (K)$.
\end{definition}
\begin{example}
Let $P = D_{2^n} = \langle h,a \,|\,h^{2^{n-1}} = h^q =a^2 =1, h^a = h^{-1} \rangle$ and $G = {\rm PSL} (2, p^d)$ for some suitable choice of $p,d$ such that the order of the Sylow $2$-subgroup of $G$ equals to $2^n$. Consider the involution of $P$, namely $h^i a$ for some $i \in \mathbb{N}$. Then the subgroup $Q_i = \lbrace 1, h^i a \rbrace$ is semi-invariant in $\mathcal{F}_{P} (G)$ if and only if $Q_i$ is invariant under any $G$-automorphism $\phi$ of any subgroup $Q_i \leq R \leq P$. In other words, $N_G (Q_i ) \geq N_G (R)$ for any   $Q_i \leq R \leq P$. Notice that $P_i  =  \langle h^2 \rangle \cdot \langle h^i a \rangle$ is a normal subgroup of $P$ containing $Q_i$, it follows from the fact $Q_i$ is not normal in $P$ that $N_P (Q_i ) \not\geq N_P (P_i)$, and so $N_G (Q_i ) \not\geq N_G (P_i)$. Thus $Q_i$ does not satisfy the $D_{2^n}$-subnormalizer condition for any $i \in \mathbb{N}$. Clearly, $P_i$ satisfies the $D_{2^n}$-subnormalizer condition for any $i \in \mathbb{N}$.
\end{example}
As a direct observation, we conclude from Lemma \ref{cover} that the semi-invariant property seems weaker than weakly $\mathcal{F}$-closed property. Hence our aim is now clear. We wonder whether the characterizations for supersolvability of saturated fusion system still hold or not if we change the condition of weakly $\mathcal{F}$-closed into the so-called weaker version named semi-invariant. Actually, using exactly the same framework of F. Aseeri and J. Kaspczyk {{\cite{FJ}}}, we find it true that all of their theorems hold if we change the condition of weakly $\mathcal{F}$-closed into semi-invariant, and hence the results of {{\cite{FJ}}} have been covered and improved. We will show the improvements in section \ref{Section 4} as our paper's main results. At this time, another question comes to our mind naturally. What does the semi-invariant property really mean in ordinary finite group theory? Or equivalently, since weakly $\mathcal{F}$-closed property means weakly closed in $S$ with respect to $G$ if $\mathcal{F}=  \mathcal{F}_S (G)$ for some $S \in {\rm Syl}_p(G)$, can we find a property in finite group theory for semi-invariant property in fusion system correspondingly? Actually, it is not hard to give the following definition, which answers the question above.
\begin{definition}
Let $G$ be a finite group, $p$ be a prime dividing the order of $G$, and $S$ be a Sylow $p$-subgroup of $G$. Now let $H$ be a subgroup of $S$, then $H$ satisfies the $S$-subnormalizer condition in $G$, if for any $H \leq K \leq S$, $N_G (H) \geq N_G (K)$.
\end{definition}
\begin{example}
Let $G  =S_4$, and $p = 2$. Then we claim that the subgroups $C_4$ and $C_2 ^2$ of the Sylow 2-subgroup $D_4$ of $G$ satisfy the $D_4$-subnormalizer condition in $S_4$. Actually, since $N_G (C_4) = D_4 = N_G (D_4)$, $N_G (C_2 ^2) = D_4 = N_G (D_4)$, we obtain that $C_2 ^2$ and $C_4$ satisfy the $D_4$-subnormalizer condition in $G = S_4$. Notice that there are two subgroups $Q_1$ and $Q_2$ of $C_2 ^2$ which are isomorphic to $C_2$, and $N_G (Q_1) = D_4$, while $N_G (Q_2) = C_2 ^2$, it follows that $Q_1$ satisfies the $D_4$-subnormalizer condition while $Q_2$ fails to satisfy the $D_4$-subnormalizer condition.
\end{example}
Correspondingly, it is obvious that the $S$-subnormalizer condition is weaker than weakly closed property in group theory. The name of $S$-subnormalizer condition comes from the well-known subnormalizer conditions. Let $G$ be a finite group and $H$ be a subgroup of $g$. Recall that $H$ is said to satisfy the subnormalizer condition in $G$, if for any $H \leq K \leq N_G (H)$, we have $N_G (K) \leq N_G (H)$. One can find that the two concepts have some commonplace, which is why we choose this name for the new generalised normality. Using our significant results in section \ref{Section 4}, we get several corollaries about the characterizations for supersolvability of $\mathcal{F}_S (G)$ under the hypothesis that certain subgroups of $G$ satisfy the $S$-subnormaizer condition in $G$, where $S$ denotes a Sylow $p$-subgroup of $G$.

In F. Aseeri and J. Kaspczyk {{\cite{FJ}}}, they pointed out the relationship between  the $p$-nilpotency of a finite group $G$  and the supersolvability of $\mathcal{F}_S (G)$. We illustrate the argument of them by the following remarks.
\begin{remark}\label{Equivalence}
Let $G$ be a finite group, $p$ be a prime dividing the order of $G$, and $S$ be a Sylow $p$-subgroup of $G$. Suppose that $H$ is a subgroup of $S$. Then $H$ is semi-invariant in $\mathcal{F}_S (G)$ if and only if for any $H \leq K \leq S$ and $\alpha \in {\rm Aut}_{\mathcal{F}_S (G)} (K)$, $H^{\alpha} = H$. In other words, $H$ is semi-invariant in $\mathcal{F}_S (G)$ if and only if for any $H \leq K \leq S$ and $g \in N_G (K)$, $H^g = H$, i.e. $g \in N_G (H)$. Hence $H$ is semi-invariant in $\mathcal{F}_S (G)$ if and only if $H$ satisfies the $S$-subnormalizer condition in $G$.
\end{remark}
\begin{remark}\label{p-nilpotence}
Let $G$ be a finite group, $p$ be a prime dividing the order of $G$, and $S$ be a Sylow $p$-subgroup of $G$. Assume that $(p-1,|G|)=1$ and $\mathcal{F}_S (G)$ is supersolvable, then we conclude from {{\cite[Theorem 1.9]{SN}}} that $G$ is $p$-nilpotent, and so $p$-supersolvable.
\end{remark}

Using the remarks above, we can actually obtain a lot of criteria for $p$-nilpotency of finite group since the supersolvability for $\mathcal{F}_S (G)$ can be transformed into the $p$-nilpotency of $G$, if $(p-1,|G|)=1$ according to Remark \ref{p-nilpotence}. Hence in Section \ref{1005}, we obtain 16 criteria for $p$-nilpotency of finite group based on 16 different kinds of generalised normalities using the results we have gained in Section \ref{1003} and Section \ref{Section 4}.
\section{Preliminaries}

Firstly, we introduce some basic lemmas about saturated fusion systems and semi-invariant property, which will play a key role in our main results. 
\begin{lemma}\label{1}
Let $p$ be a prime and $\mathcal{F}$ be a saturated fusion system on a finite $p$-group $S$. Assume that the fusion system $N_{\mathcal{F}} (Q)$ is supersolvable for any $Q \in \mathcal{E}_{\mathcal{F}} ^{*}$, then $\mathcal{F}$ is supersolvable.\end{lemma}
\begin{proof}
Let $Q \in \mathcal{E}_{\mathcal{F}} ^{*}$. It follows from $N_{\mathcal{F}} (Q)$ is supersolvable and {{\cite[Proposition 1.3]{SN}}} that ${\rm Aut}_{N_{\mathcal{F}} (Q)} (Q) = {\rm Aut}_{\mathcal{F}} (Q)$ is $p$-closed. Hence, ${\rm Out}_{\mathcal{F}} (Q)$ is $p$-closed since ${\rm Out}_{\mathcal{F}} (Q)$ is a quotient group of ${\rm Aut}_{\mathcal{F}} (Q)$. By {{\cite[Proposition A.7 (c)]{AK}}}, we conclude that there is no subgroup $H$ of a $p$-closed finite group $G$ such that $H$ is strongly $p$-embedded with respect to $G$. Therefore ${\rm Out}_{\mathcal{F}} (Q)$ does not possess a strongly $p$-embedded subgroup, which implies that $Q$ is not $\mathcal{F}$-essential. Thus we get that $\mathcal{E}_{\mathcal{F}} ^{*} = \lbrace S \rbrace$. Now it indicates from {{\cite[Part I, Proposition 4.5]{AK}}} that $S$ is normal in $\mathcal{F}$. Hence the proof is complete since $N_{\mathcal{F}} (S) = \mathcal{F}$ is supersolvable by our hypothesis.
\end{proof}

\begin{lemma}[{{\cite[Lemma 2.9]{FJ}}}]\label{2}
Let $G$ be a finite group, $p \in \pi(G)$, and $S$ be a Sylow $p$-subgroup of $G$. Suppose that for any proper subgroup $H$ of $G$ with $O_p (G) <S \cap H$ and $S \cap H \in {\rm Syl}_p (H)$, the fusion system $\mathcal{F}_{S \cap H} (H)$ is supersolvable. Assume additionally that $O_p (G) \leq Z_{\mathfrak{U}} (G)$. Then $\mathcal{F}_{S} (G)$ is supersolvable.
\end{lemma}

\begin{lemma}[{{\cite[Lemma 2.10]{FJ}}}]\label{6}
Let $p$ be a prime number and $\mathcal{F}$ be a saturated fusion system over a $p$-group $S$. Assume that for any subgroup $R$ such that $O_p (\mathcal{F}) <R$, any proper saturated fusion subsystem of $\mathcal{F}$ on $R$ is supersolvable. Suppose moreover that there is a series $1 = S_0 \leq S_1 \leq \cdots \leq S_n = O_p (\mathcal{F})$ such that $S_i$ is strongly $\mathcal{F}$-closed, $i=0,1,\cdots,n$, $S_{i+1}/S_i$ is cyclic, $i=0,1,\cdots,n-1$. Then $\mathcal{F}$ is supersolvable.
\end{lemma}
\begin{lemma}\label{7}
Let $p$ be a prime number and $\mathcal{F}$ be a fusion system on a finite $p$-group $S$. Assume that $Q$ is a strongly closed subgroup of $S$ in $\mathcal{F}$, $R$ is a semi-invariant subgroup of $S$ in $\mathcal{F}$, and $Q \leq R \leq S$. Then $R/Q$ is semi-invariant in $\mathcal{F}/Q$.
\end{lemma}
\begin{proof}
Let $H/Q$ be a subgroup of $S/Q$ such that $R/Q \leq H/Q \leq S/Q$, and $\alpha \in {\rm Aut}_{\mathcal{F/Q}} (H/Q)$. By the definition of $\mathcal{F}/Q$, there exists $\overline{\alpha} \in {\rm Aut}_{\mathcal{F }} (H)$, such that $(hQ)^{\alpha} = h^{\overline{\alpha}}Q$ for all $hQ \in H/Q$. Since $R$ is semi-invariant in $\mathcal{F}$, $R^{\overline{\alpha}} = R$. Hence we conclude that $(R/Q)^{\alpha} = \lbrace r^{\overline{\alpha}} Q\,|\,r \in R \rbrace= R/Q$. By the randomness of $H$ and $\alpha$, $R/Q$ is semi-invariant in $\mathcal{F}/Q$ and we are done.
\end{proof}
\begin{lemma}[{{\cite[Lemma 1.11]{BO}}}]\label{8}
Let $p$ be a prime, $A$ be an abelian $p$-group and $G$ be a subgroup of ${\rm Aut} (A)$. Assume that the following hold: 
\begin{itemize}
\item[(1)] The Sylow $p$-subgroups of $G$ has order $p$ and $G$ is not $p$-closed.
\item[(2)] For each $x \in G$ of order $p$, $[x,A]$ has order $p$.
\end{itemize}
Set $A:=O_p (\mathcal{F})$, $G:={\rm Aut} (O_P (\mathcal{F}))$, $A_1 = C_A (H)$, $H = O^{p'} (G)$ and $A_2 = [H,A]$. Then $G$ normalizes $A_1$ and $A_2$, $A = A_1 \times A_2$, $A_2 = C_p \times C_p$.

\end{lemma}

\begin{lemma}\label{9}
Let $p$ be a prime, $\mathcal{F}$ be a supersolvable saturated fusion system over a $p$-group $S$. Suppose that $Q$ is a subgroup of $S$ such that $Q$ is semi-invariant in $\mathcal{F}$. Then $Q$ is strongly closed in $\mathcal{F}$. 
\end{lemma}
\begin{proof}
Let $P = \langle x \rangle$ be a cyclic subgroup of $Q$, $R = \langle y \rangle$ be a cyclic subgroup of $S$ and $\phi \in {\rm Iso}_{\mathcal{F}} (P,R)$ such that $x^{\phi} = y$. By {{\cite[Lemma 1.11]{SN}}}, $O_p (\mathcal{F}) = S$. Hence there exists an extension of $\phi$, namely $\overline{\phi}$, such that $\overline{\phi} \in {\rm Hom}_{\mathcal{F}} (S P,SR) = {\rm Aut}_{\mathcal{F}} (S)$. Since $Q$ is semi-invariant in $\mathcal{F}$ and $P \leq Q$, we conclude that $R = P^{\phi} = P^{\overline{\phi}} \leq Q^{\overline{\phi}}=Q$. Therefore by the randomness of $x$ we assert that $Q$ is strongly $\mathcal{F}$-closed.
\end{proof}
\begin{lemma}\label{10}
Let $p$ be a prime, $\mathcal{F}$ be a fusion system on a finite $p$-group $S$ and $Q \leq R \leq S$. Assume that $Q \unlhd \mathcal{F}$ and that $R/Q$ is strongly $\mathcal{F}/Q$-closed, then $R$ is strongly $\mathcal{F}$-closed.
\end{lemma}
\begin{proof}
Let $P = \langle x \rangle$ be a cyclic subgroup of $R$, $T = \langle y \rangle$ be a subgroup of $S$, and $\phi \in {\rm Iso}_{\mathcal{F}} (P,T)$ such that $x^{\phi} = y$. Since $Q$ is normal in $\mathcal{F}$, $\phi$ can extend to $\psi \in {\rm Hom}_{\mathcal{F}} (PQ,TQ)$ such that $Q^{\psi} = Q$. Now we define the following morphism:
$$\overline{\psi}:PQ/Q \rightarrow TQ/Q,\,xQ \mapsto x^{\psi} Q.$$
Then clearly $\overline{\psi} \in {\rm Hom}_{\mathcal{F}/Q} (PQ/Q,TQ/Q)$. Since $R/Q$ is strongly $\mathcal{F}/Q$-closed and $PQ/Q \leq R/Q$, we conclude that $(PQ/Q)^{\overline{\psi}} = P^{\psi}Q/Q=P^{\phi}Q/Q = TQ/Q \leq R/Q$, which yields that $T \leq R$. By the randomness of $P$, we conclude that $R$ is strongly $\mathcal{F}$-closed and we are done.
\end{proof}
\begin{lemma}\label{cover}
Let $p$ be a prime, $S$ be a $p$-group and $\mathcal{F}$ be a fusion system over $S$. Suppose that $H$ is a subgroup of $S$ such that $H$ is weakly $\mathcal{F}$-closed, then $H$ is semi-invariant in $\mathcal{F}$.
\end{lemma}
\begin{proof}
Suppose that the lemma is not true and let $H \leq S$ be a counterexample. Then there exists a subgroup $H \leq N \leq S$ and $\alpha \in {\rm Aut}_{\mathcal{F}} (N)$ such that $H^{\alpha} \neq H$. Since the morphism $\phi:H \rightarrow N$, $h \mapsto h$ is an $\mathcal{F}$-morphism, we conclude that $\phi \circ \alpha :H\rightarrow N, h \mapsto h^{\alpha}$ is an $\mathcal{F}$-morphism. Hence the morphism $\psi:H \rightarrow H^{\alpha}, h \rightarrow h^{\alpha}$ is an $\mathcal{F}$-morphism. Note that $\psi $ is in fact an isomorphism, hence $H$ and $H^{\alpha}$ is $\mathcal{F}$-conjugate. However, it follows from $H \neq H^{\alpha}$ that $H$ is not weakly $\mathcal{F}$-closed, a contradiction. Hence our proof is complete.
\end{proof}
As a direct observation, we conclude from Lemma \ref{cover} that the semi-invariant property seems weaker than weakly $\mathcal{F}$-closed property. Hence our aim is now clear. We wonder whether the characterizations for supersolvability of saturated fusion system still hold or not if we change the condition of weakly $\mathcal{F}$-closed into the so-called weaker version named semi-invariant.
\begin{lemma}[{{\cite[Lemma 2.4 (1)]{FJ}}}]\label{12}
Let $p$ be a prime, $S$ be a $p$-group and $\mathcal{F}$ be a  supersolvable saturated fusion system over $S$ and $Q$ be a strongly $\mathcal{F}$-closed subgroup of $S$. Then there exists a series 
$$1 =S_0 \leq S_1 \leq \cdots \leq S_m = S,$$
such that $S_i$ is strongly $\mathcal{F}$-closed, $i=0,1,\cdots,m$, $S_{i+1} / S_i$ is cyclic, $i=0,1,\cdots,m-1$, and $S_n = Q$ for some $0 \leq n \leq m$.
\end{lemma}
\begin{lemma}\label{13}
Let $G$ be a minimal non-nilpotent group, and $S$ be a Sylow subgroup of $G$ which is not normal in $G$. Then $S$ is cyclic.
\end{lemma}
\begin{proof}
It follows from {{\cite[Kapitel III, Satz 5.2]{HU}}} that $G = P \rtimes Q$, where $P$ is the normal Sylow $p$-subgroup of $G$, $Q$ is a cyclic Sylow $q$-subgroup of $G$, $p \neq q$ are distinct prime numbers. Since $S$ is not normal in $G$, we assert immediately that $S$ is cyclic.
\end{proof}
\begin{lemma}[{{\cite[Lemma 2.12 (2)]{FJ}}}]\label{14}
Let $p$ be an odd prime and $t =2 $ or $t=3$. If $H$ is a non-trivial cyclic subgroup of $GL_t (p)$, then $H$ is cyclic of order $p$.

\end{lemma}
Let $p$ be a prime and $S$ be a $p$-group. Recall that $S$ is said to be an $\mathcal{A}_2$-group, if $S$ contains a non-abelian maximal subgroup, and all subgroups of $S$ of index $p^2$ are abelian. According to {{\cite{YZ}}}, we have the following lemma.
\begin{lemma}\label{15}
Let $p$ be a prime and $S$ be a $p$-group. Suppose that $S$ is a non-metacyclic $\mathcal{A}_2$-group with $|S|>p^4$. Then the following hold:
\begin{itemize}
\item[(1)]{{\cite[Proposition 71.4]{YZ}}} Assume that $S$ has exactly one abelian maximal subgroup and $S' \leq Z(S)$, then $Z(S) = \Phi (S)$.
\item[(2)]{{\cite[Proposition 71.5]{YZ}}} If any maximal subgroup of $S$ is non-abelian and $p$ is odd, then $|S| = p^5$.
\end{itemize}
\end{lemma}

Now we would like to introduce the lemmas about the generalised normalities, which are useful for our characterizations for $p$-nilpotency of finite groups and supersolvability for $\mathcal{F}_S (G)$ in Section \ref{1003} and Section \ref{1005}.\begin{lemma}\label{16}
Let $p$ be a prime such that $p$ is a prime divisor of $G$, $S$ be a Sylow $p$-subgroup of $G$, and $H$ be a subgroup of $S$. Then the following statements are equivalent. 
\begin{itemize}
\item[(1)] $H$ is pronormal in $G$.
\item[(2)] $H$ is weakly normal in $G$.
\item[(3)] $H$ is weakly closed in $G$ with respect to $S$.
\item[(4)] $N_G (H)$ is the subnormalizer of $H$ in $G$.
\end{itemize}
\end{lemma}
\begin{proof}
It follows directly from {{\cite[Theorem 2]{PE}}} and {{\cite[Theorem 4.3]{LY}}}.
\end{proof}
\begin{lemma}[{{\cite[Theorem 3.5]{GL3}}}]\label{3}
Let $G$ be a finite group and $P$ be a normal $p$-subgroup of $G$, where $p$ is a prime divisor of $G$. Assume that $P$ has a subgroup $D$ such that $1 < |D| < |P|$, and any subgroup of $P$ with order $|D|$ or $4$ (if $P$ is a non-abelian 2-group and $|D|=2$) is an ICC-subgroup of $G$, then $P \leq Z_{\mathfrak{U}} (G)$.
\end{lemma}
\begin{lemma}[{{\cite[Proposition 3.1]{WS}}}]\label{4}
Let $P$ be a non-identity normal subgroup of $G$ with $|P|=p^n$, where $p$ is a prime divisor of $|G|$. Suppose that $n>1$, and there exists an integer $n>k \geq 1$ that any subgroup of $P$ of order $p^k$ and any cyclic subgroup of $P$ of order $4$ (if $P$ is a non-abelian 2-group and $n-1>k=1$) are $\mathfrak{U}$-embedded in $G$. Then $P \leq Z_{\mathfrak{U}} (G)$.
\end{lemma}
\begin{lemma}[{{\cite[Theorem]{SK}}}]\label{5}
Let $p$ be a prime and $P$ be a non-trivial normal $p$-subgroup of $G$. Suppose that there exists a subgroup $D$ of $P$ with $1 <|D|<|P|$ such that every subgroup of $P$ of order $|D|$ is normal in $G$. If $P$ is a non-abelian $2$-group and $|D|=2$, suppose moreover that   every cyclic subgroup of $P$ of order $4$ is normal in $G$. Then $P \leq Z_{\mathfrak{U}} (G)$.
\end{lemma}
\begin{lemma}[{{\cite[Main Theorem]{LI}}}]\label{17}
Let $H$ be a normal subgroup of a group $G$.  Assume that for any non-cyclic Sylow subgroup $P$ of $H$, either all maximal subgroups of $P$ or all cyclic subgroups of $P$ of prime order and order $4$ are $E$-supplemented {{\cite[Definition 1.1]{LI}}} in $G$. Then each
chief factor of $G$ below $H$ is cyclic.
\end{lemma}
\begin{lemma}\label{18}
Let $G$ be a finite group and $H$ be an $ICSC$-subgroup  {{\cite[Definition 1.1]{GL3}}} of $G$. Then for any $H \leq K \leq G$, $H$ is an $ICSC$-subgroup of $K$.
\end{lemma}
\begin{proof}
Since $H$ is an $ICSC$-subgroup of $G$, we conclude that $H \cap [H,G] \leq H_{scG}$, where $H_{scG}$ is a semi-$CAP$-subgroup of $G$ contained in $H$. By lemma {{\cite[Lemma 2.1]{BB}}}, $H_{scG}$ is a semi-$CAP$-subgroup of $K$ contained in $H$. Hence we can write $H_{scG}$ as $H_{scK}$. Then we obtain that $H \cap [H,K] \leq H \cap [H,G] \leq H_{scG} =  H_{scK}$, which implies that $H$ is an $ICSC$-subgroup of $K$. 
\end{proof}
\begin{lemma}[{{\cite[Theorem 3.1]{GL3}}}]\label{19}
Let $G$ be a group and $ P$ a normal $p$-subgroup of $G$, where $p$ is a prime divisor of $|G|$. Suppose that every cyclic subgroup of $P$ with order $p$ and 4 (if $P$ is a non-abelian 2-group) is an $ICSC$-subgroup of $G$, then $P \leq Z_{\mathfrak{U}} (G)$.
\end{lemma}
\begin{lemma}[{{\cite[Theorem 2.3]{BB}}}]\label{20}
Let $G$ be a finite group, $p$ be a prime divisor of $|G|$, and $P$ be a normal $p$-subgroup of $G$. Assume that every cyclic subgroup of $P$ with order $p$ and 4 (if $p=2$) is a partial $CAP$-subgroup {{\cite{BB}}} of $G$, then $P \leq Z_{\mathfrak{U}} (G)$.
\end{lemma}
\begin{lemma}[{{\cite[Theorem 3.2]{AS2}}}]\label{21}
Let $p$ be an odd prime and $P$ be a normal subgroup of $G$. Suppose that $P$ has a subgroup $D$ such that $1 \leq D <P$, and all subgroups of $P$ of order $|D|$ or $p|D|$ are  $c$-supplemented {{\cite{BB2}}} in $G$, then $P \leq Z_{\mathfrak{U}} (G)$.
\end{lemma}
\begin{lemma}\label{22}
Let $G$ be a finite group and $H$ be a subgroup of $G$. Suppose that $H$ is $c$-supplemented in $G$, then $H$ is $c$-supplemented in $K$ for any $H \leq K \leq G$.
\end{lemma}
\begin{proof}
Since $H$ is $c$-supplemented in $G$, there exists a subgroup $T$ of $G$ such that $G = TH$ and $T \cap H \leq H_G$. Hence we conclude from Dedekind Modular Law that $K = K \cap TH = H(T \cap K)$. Then it follows that $H \cap T \cap K \leq H \cap T \leq H_G \leq H_K$, which implies that $H$ is $c$-supplemented in $K$ and we are done.
\end{proof}
\begin{lemma}[{{\cite[Lemma 2.2]{CZ}}}]\label{23}
Let $G$ be a finite group and $K$ be a subgroup of $G$. Assume that $K$ is a weakly $\mathcal{H}$-subgroup {{\cite{AS4}}} of $G$, then for any $K \leq T \leq G$, $K$ is a weakly $\mathcal{H}$-subgroup of $T$.
\end{lemma}
\begin{lemma}[{{\cite[Theorem 3.1]{CZ}}} and {{\cite[Lemma 2.6]{LQ}}}]\label{24}
Let $P$ be a normal $p$-subgroup of a group $G$, where $p$ is a prime divisor of $|G|$. Suppose that one of the following holds:
\begin{itemize}
\item[(1)] Every cyclic subgroup
of $P$ of order $p$ or $4$ (if $P$ is a non-cyclic $2$-group) either is a weakly $\mathcal{H}$-subgroup of
$G$ or has a supersoluble supplement in $G$.
\item[(2)] $p$ is odd, and there exists a subgroup $D$ of $P$ such that every subgroup of $P$ of order $|D|$ is a weakly $\mathcal{H}$-subgroup of $G$.
\end{itemize}
Then $P \leq Z_{\mathfrak{U}} (G)$.
\end{lemma}
\begin{lemma}\label{25}
Let $G$ be a finite group and $H$ be a subgroup of $G$. Suppose that $H$ is a weakly $S\Phi$-supplemented subgroup {{\cite[Definition 1.1]{WU}}} of $G$, then $H$ is a weakly $S\Phi$-supplemented subgroup of $T$ for any $H \leq T \leq G$.
\end{lemma}
\begin{proof}
As  $H$ is a weakly $S\Phi$-supplemented subgroup of $G$, there exists a subgroup $K$ of $G$ such that $G = KH$ and $K \cap H \leq   \Phi (H) H_{sG}$, where $H_{sG}$ is generated by all subgroups of $H$ which are $s$-permutable in $G$. Then we conclude from {{\cite[Lemma 2.1 (b)]{LL}}} that $H_{sG} \leq H_{sT}$. Hence we assert from Dedekind modular law that $T = T \cap HK  = H (T \cap K)$. Since $H \cap K  \cap T \leq H \cap K \leq H_{sG} \leq H_{sT}$, it yields that $H$ is a weakly $S\Phi$-supplemented subgroup of $T$ and the result follows.
\end{proof}
\begin{lemma}[{{\cite[Theorem 1.4]{AS3}}}]\label{26}
Let $G$ be a finite group, $p$ be a prime divisor of $|G|$ and $P$ be a normal subgroup of $G$. Assume that there exists a subgroup $D$ of $P$ of order $1 \leq |D| <|P|$ such that every subgroup of $P$ of order $|D|$ or $p|D|$ is a weakly $S\Phi$-supplemented subgroup of $G$. If $P$ is a non-abelian $2$-group and $|D|=1$, we have additionally that every cyclic subgroup of $G$ of order $4$ is a weakly $S\Phi$-supplemented subgroup of $G$. Then $P \leq Z_{\mathfrak{U}}(G)$.
\end{lemma}
\begin{lemma}[{{\cite[Theorem 1.1]{WH}}}]\label{27}
Let $E$ be a normal subgroup of $G$. If for every non-cyclic Sylow $p$-subgroup $P$ of $E$, there exists a subgroup $D$ of $P$  with $1 <|D|<|P|$ such that every subgroup of $P$ of order $|D|$ or $4$ (if $p=|D|=2$) is an $\mathscr{H} C$-subgroup {{\cite{GX}}} of $G$, then $E \leq Z_{\mathfrak{U}} (G)$.
\end{lemma}
\begin{lemma}[{{\cite[Lemma 3.2]{LF}}}]\label{28}
Let $E$ be a normal subgroup of $G$. Suppose that there exists a  normal subgroup $X$ of $G$ such that $F^{*} (E) \leq X \leq E$, and $X$ satisfies the following properties: for every non-cyclic Sylow subgroup $P$ of $X$, there exists a subgroup $D$ of $P$ such that $1<|D|<|P|$, and every subgroup $H$ of $P$ with order $|D|$ or $2|D|$ (if $P$ is a non-abelian 2-group and $|P:D|>2$) is either $S$-quasinormally embedded {{\cite{XX}}} or $SS$-quasinormal {{\cite{LSK}}} in $G$, then $E \leq Z_{\mathfrak{U}} (G)$.
\end{lemma}
\begin{lemma}[{{\cite[Lemma 2.2]{HX}}}]\label{29}
Let $P$ be a non-trivial normal subgroup of a finite group $G$. If there exists a subgroup $D$ of $P$ such that $1<|D|<|P|$, and any subgroup of $P$ of order $|D|$ or $p|D|$ is weakly $\mathcal{HC}$-embedded {{\cite{LF}}} in $G$, then $P \leq Z_{\mathfrak{U}} (G)$.
\end{lemma}
\begin{lemma}[{{\cite[Proposition 1.6]{MCG}}}]\label{30}
Let $G$ be a finite group and $P$ be a normal $p$-subgroup of $G$. Suppose that $P$ has a subgroup $D$ with $1<|D|<|P|$ such that every subgroup of $P$ of order $|D|$ and every cyclic subgroup of $P$ of order 4 (if $P$ is a non-abelian 2-group and $|D|=2$) are generalised $S\Phi$-supplemented {{\cite{MCG}}} in $G$, then $P \leq Z_{\mathfrak{U}} (G)$.
\end{lemma}
\begin{lemma}[{{\cite[Theorem 1.3]{LZ}}}]\label{31}
Let $G$ be $Q_8$-$free$ and $P$ be a non-trivial normal $p$-subgroup of $G$. If all subgroups of $P$ of order $p$ are $NE^{*}$-subgroups {{\cite{LZ}}} of $G$, then $P \leq Z_{\mathfrak{U}} (G)$.
\end{lemma}
\section{Characterizations for supersolvability of $\mathcal{F}_S (G)$}\label{1003}
It is well known that many researchers have generalized the normality in finite groups, and obtain a variety of characterizations for $p$-nilpotency. In this section, we first give two alternative ways to axiomize the characterizations for supersolvability of $\mathcal{F}_S (G)$. Then, we give some examples of characterizations for supersolvability of $\mathcal{F}_S (G)$ under the assumption that certain subgroups of $G$ satisfy some kinds of generalized normalities. 

As to axiomize the characterizations for supersolvability of $\mathcal{F}_S (G)$, we need to summarize the similarities of certain generalized normalities, and give them an unified definition. Actually, almost all characterizations for supersolvability of $\mathcal{F}_S (G)$ can be divided into the following 12 cases, and each case provides a method to generalize the characterizations for supersolvability of $\mathcal{F}_S (G)$ under the condition that some  specialized subgroups of $S$ satisfy certain generalised normalities. More precisely, we abstract the generalised normalities into approximately 12 cases, and for each case we use different methods to research the structure of $\mathcal{F}_S (G)$ under the assumption that certain subgroups of $S$ satisfy a generalised version of generalised normalities, namely relation. As a result, we obtain 12 theorems describing the structure of fusion systems of finite groups, and use 12 different examples of generalised normalities to present the applications of those 12 theorems. In fact, almost every characterizations for supersolvability of $\mathcal{F}_S (G)$ can use the similar method to get access to, and what we have done is to give an inspiration to explore the universal language to describe the characterizations for supersolvability of $\mathcal{F}_S (G)$. 

Now, we are going to introduce one alternative way to generalise the characterizations for supersolvability of $\mathcal{F}_S (G)$, and first we would like to introduce the definitions of 12 different  relations of subgroups of  finite groups.
\begin{definition}\label{eternal}
Let $G$ be an arbitrary finite group and $ \rightsquigarrow$ be a relation between the subgroups of $G$ and $G$. Then $\rightsquigarrow$ is said to be 1-eternal, if for any finite group $G$, (0) and (1) hold;  $\rightsquigarrow$ is said to be 2-eternal, if for any finite group $G$, (0) and (2) hold;  $\rightsquigarrow$ is said to be 3-eternal, if for any finite group $G$, (0) and (3) hold;  $\rightsquigarrow$ is said to be 4-eternal, if for any finite group $G$, (0) and (4) hold;  $\rightsquigarrow$ is said to be 5-eternal, if for any finite group $G$, (0) and (5) hold;  $\rightsquigarrow$ is said to be 6-eternal, if for any finite group $G$, (0) and (6) hold;  $\rightsquigarrow$ is said to be 7-eternal, if for any finite group $G$, (0) and (7) hold;  $\rightsquigarrow$ is said to be 8-eternal, if for any finite group $G$, (0) and (8) hold; $\rightsquigarrow$ is said to be 9-eternal, if for any finite group $G$, (0) and (9) hold.  Also, $\rightsquigarrow$ is said to be semi-1-eternal, if for any finite group $G$, (0) and (1') hold;  $\rightsquigarrow$ is said to be semi-2-eternal, if for any finite group $G$, (0) and (2') hold; $\rightsquigarrow$ is said to be semi-3-eternal, if for any finite group $G$, (0) and (3') hold.
\begin{itemize}
\item[(0)] For any subgroups $H \leq K \leq G$, if $H\rightsquigarrow G$, then $H\rightsquigarrow K$.
\item[(1)] For any $p$ a prime divisor of $|G|$, if there exists a subgroup $D$ of $O_p (G)$ such that any subgroup $H$ of $O_p (G)$ of order $1<|D|<|O_p (G)|$ satisfies the relation $H \rightsquigarrow G$, and every cyclic subgroup of $O_p (G)$ of order $4$ satisfies the relation $H \rightsquigarrow G$ (if $O_p (G)$ is non-abelian and $|D|=2$), then $O_p (G) \leq Z_{\mathfrak{U}} (G)$.
\item[(2)] For any $p$ a prime divisor of $|G|$, if there exists a subgroup $D$ of $O_p (G)$ such that any subgroup $H$ of $O_p (G)$ of order $1<|D|<|O_p (G)|$ satisfies the relation $H \rightsquigarrow G$, and every cyclic subgroup of $O_p (G)$ of order $4$ satisfies the relation $H \rightsquigarrow G$ (if $O_p (G)$ is non-abelian, $|O_p (G)| >4$ and $|D|=2$), then $O_p (G) \leq Z_{\mathfrak{U}} (G)$.
\item[(3)] For any $p$ a prime divisor of $|G|$, if there exists a subgroup $D$ of $O_p (G)$ such that any subgroup $H$ of $O_p (G)$ of order $1<|D|<|O_p (G)|$ satisfies the relation $H \rightsquigarrow G$, and every cyclic subgroup of $O_p (G)$ of order $4$ satisfies the relation $H \rightsquigarrow G$ (if $O_p (G)$ is non-abelian  and $p=2$), then $O_p (G) \leq Z_{\mathfrak{U}} (G)$.
\item[(4)] For any $p$ a prime divisor of $|G|$, if there exists a subgroup $D$ of $O_p (G)$ such that any subgroup $H$ of $O_p (G)$ of order $1<|D|<|O_p (G)|$ satisfies the relation $H \rightsquigarrow G$, and every cyclic subgroup of $O_p (G)$ of order $4$ satisfies the relation $H \rightsquigarrow G$ (if   $p=2$), then $O_p (G) \leq Z_{\mathfrak{U}} (G)$.
\item[(5)] For any $p$ an odd prime divisor of $|G|$, if there exists a subgroup $D$ of $O_p (G)$ such that any subgroup $H$ of $O_p (G)$ of order $1\leq |D|<|O_p (G)|$ or $p|D|$ satisfies the relation $H \rightsquigarrow G$, then $O_p (G) \leq Z_{\mathfrak{U}} (G)$.
\item[(6)] For any $p$ an odd prime divisor of $|G|$, if there exists a subgroup $D$ of $O_p (G)$ such that any subgroup $H$ of $O_p (G)$ of order $1\leq |D|<|O_p (G)|$   satisfies the relation $H \rightsquigarrow G$, then $O_p (G) \leq Z_{\mathfrak{U}} (G)$.
\item[(7)] For any $p$ a prime divisor of $|G|$, if there exists a subgroup $D$ of $O_p (G)$ such that any subgroup $H$ of $O_p (G)$ of order $1\leq |D|<|O_p (G)|$ or $p|D|$ satisfies the relation $H \rightsquigarrow G$, and every cyclic subgroup of $O_p (G)$ of order $4$ satisfies the relation $H \rightsquigarrow G$ (if   $|D|=2=p$),  then $O_p (G) \leq Z_{\mathfrak{U}} (G)$.
\item[(8)] For any $p$ a prime divisor of $|G|$, if there exists a subgroup $D$ of $O_p (G)$ such that any subgroup $H$ of $O_p (G)$ of order $1< |D|<|O_p (G)|$  satisfies the relation $H \rightsquigarrow G$, and every cyclic subgroup of $O_p (G)$ of order $2 |D|$ satisfies the relation $H \rightsquigarrow G$ (if $O_p (G)$ is a non-abelian 2-group and $|O_p (G) :D|>2$),  then $O_p (G) \leq Z_{\mathfrak{U}} (G)$.
\item[(9)] For any $p$ a  prime divisor of $|G|$, if there exists a subgroup $D$ of $O_p (G)$ such that any subgroup $H$ of $O_p (G)$ of order $1< |D|<|O_p (G)|$ or $p|D|$ satisfies the relation $H \rightsquigarrow G$, then $O_p (G) \leq Z_{\mathfrak{U}} (G)$.
\item[(1')] For any $p$ a prime divisor of $|G|$, if any subgroup $H$ of $O_p (G)$ of order $p$ satisfies the relation $H \rightsquigarrow G$, and every cyclic subgroup of $O_p (G)$ of order $4$ satisfies the relation $H \rightsquigarrow G$ (if $O_p (G)$ is non-abelian and $p=2$), then $O_p (G) \leq Z_{\mathfrak{U}} (G)$.
\item[(2')] For any $p$ a prime divisor of $|G|$, if any subgroup $H$ of $O_p (G)$ of order $p$ satisfies the relation $H \rightsquigarrow G$, and every cyclic subgroup of $O_p (G)$ of order $4$ satisfies the relation $H \rightsquigarrow G$ (if  $p=2$), then $O_p (G) \leq Z_{\mathfrak{U}} (G)$.
\item[(3')] For any $p$ a prime divisor of $|G|$, if any subgroup $H$ of $O_p (G)$ of order $p$ satisfies the relation $H \rightsquigarrow G$, and every cyclic subgroup of $O_p (G)$ of order $4$ satisfies the relation $H \rightsquigarrow G$ (if  $P$ is a non-cyclic 2-group), then $O_p (G) \leq Z_{\mathfrak{U}} (G)$.
\end{itemize}
\end{definition}

With the definition above, we are now ready to obtain the following theorem, as a way to generalize the work of many specialized results which focus on only one or two normalities.
\begin{theorem}\label{3.1}
Let $\rightsquigarrow$ be a 1-eternal relation, $G$ be a finite group, $p$ be a prime divisor of $|G|$, and $S$ be a Sylow $p$-subgroup of $G$. Suppose that there exists a subgroup $1<D<S$ such that any subgroup of $S$ 
with order $|D|$ or $p|D|$ is abelian, any subgroup $H$ of $S$ of order $|D|$ or any cyclic subgroup $H$ of $S$ of order $4$ (If $S$ is a non-abelian $2$-group and $|D| =2$) satisfies the relation $H \rightsquigarrow G$, then $\mathcal{F}_S (G)$  is supersolvable.
\end{theorem}
\begin{proof}
Assume that the theorem is false, and let $G$ be a counterexample of  minimal order. Now denote $\mathcal{F}_S (G)$ by $\mathcal{F}$.
\begin{itemize}
\item[\textbf{Step 1.}] Let $H$ be a proper subgroup of $G$ such that $S \cap H \in {\rm Syl}_p (H)$ and $|S \cap H| \geq p |D|$. Then $\mathcal{F}_{S \cap H} (H)$ is supersolvable.
\end{itemize}

By our hypothesis, every subgroup $T$ of $S \cap H$ with order $|D|$ or $4$ (If $S \cap H$ is a non-abelian $2$-group and $|D| =2$)  satisfies the relation $T \rightsquigarrow G$. Then every subgroup $T$ of $S \cap H$ with order $|D|$ or any cyclic subgroup $T$ of $S \cap H$ of order $4$ (If $S \cap H$ is a non-abelian $2$-group and $|D| =2$) satisfies the relation $T \rightsquigarrow H$. Clearly any subgroup of $S \cap H$ with order $|D|$ or $p|D|$ is abelian. Since $1<|D| <|S \cap H|$, $H$ satisfies the hypothesis and it follows from the minimal choice of $G$ that $\mathcal{F}_{S \cap H} (H)$ is supersolvable.
\begin{itemize}
\item[\textbf{Step 2.}] Let $Q \in \mathcal{E}_{\mathcal{F}} ^{*}$, then $|Q| \geq p|D|$. If moreover that $Q \not\unlhd G$, then $N_{\mathcal{F}} (Q)$ is supersolvable.
\end{itemize}

Suppose that there exists a subgroup $Q \in \mathcal{E}_{\mathcal{F}} ^{*}$ such that $|Q| < p|D|$. Then there is a subgroup $R$ of $S$ such that $|R| = p |D|$, and $Q <R$. By our hypothesis, $R$ is an abelian subgroup of $S$, and therefore $R \leq C_S (Q)$. Since $Q < R \leq S$, we conclude from $Q$ is a member of $\mathcal{E}_{\mathcal{F}} ^{*}$ that $Q$ is $\mathcal{F}$-essential. By the definition, $Q$ is $\mathcal{F}$-centric. Hence $R \leq C_S (Q) =Z(Q) \leq Q$, a contradiction. Thus $|Q| \geq p|D|$.

Assume that $Q$ is not normal in $G$. Therefore $N_G (Q)$ is a proper subgroup of $G$. Since $Q \in \mathcal{E}_{\mathcal{F}} ^{*}$, $Q$ is fully $\mathcal{F}$-normalized or $Q=S$. Clearly $S$ is fully $\mathcal{F}$-normalized, hence $Q$ is always fully $\mathcal{F}$-normalized. By the argument below {{\cite[Definition 2.4]{AK}}}, $S \cap N_G (Q) = N_S (Q) \in {\rm Syl}_p (N_G(Q))$. Since $|N_S (Q)| \geq |Q| \geq p|D|$, it yields that $N_G (Q)$ satisfies the hypothesis of Step 1, and so $\mathcal{F}_{N_S (Q)} (N_G(Q))=N_{\mathcal{F}} (Q)$ is supersolvable.
\begin{itemize}
\item[\textbf{Step 3.}] $|O_p (G)| \geq p|D|$.
\end{itemize}

Assume that there does not exist a subgroup $Q \in \mathcal{E}_{\mathcal{F}} ^{*}$ such that $Q \unlhd G$. Then for each $Q \in  \mathcal{E}_{\mathcal{F}} ^{*}$, the fusion system $N_{\mathcal{F}} (Q)$ is supersolvable by Step 2. By Lemma \ref{1}, $\mathcal{F}$ is supersolvable, a contradiction. Thus there exists a subgroup $Q \in \mathcal{E}_{\mathcal{F}} ^{*}$ such that $Q \unlhd G$. Hence we conclude from Step 2 that $|O_p (G)| \geq |Q| \geq p|D|$.
\begin{itemize}
\item[\textbf{Step 4.}] $O_p (G) \leq Z_{\mathfrak{U}} (G)$.
\end{itemize}

It follows from $|O_p (G)| \geq p|D|$ that there exists a subgroup $D_0 \leq O_p (G)$ with order $|D|$ such that any subgroup $T$ of $O_p (G)$ of order $|D|$ or any cyclic subgroup $T$ of $O_p (G)$ of order $4$ (If $O_p (G)$ is a non-abelian $2$-group and $|D| =2$) satisfies the relation $T \rightsquigarrow G$. Since $\rightsquigarrow$ is a 1-eternal relation, it yields that $O_p (G) \leq Z_{\mathfrak{U}} (G)$ and this part is complete.
\begin{itemize}
\item[\textbf{Step 5.}] Final contradiction.
\end{itemize}

Suppose that $H$ is a proper subgroup of $G$ such that $O_p (G) < S \cap H$ and $S \cap H \in {\rm Syl}_p (H)$. By Step 1 and Step 3, $|S \cap H| >|O_p (G)| \geq p|D|$ and so $\mathcal{F}_{S \cap H} (H)$ is supersolvable. Since $O_p (G) \leq Z_{\mathfrak{U}} (G)$ by Step 4, it follows directly from Lemma \ref{2} that $\mathcal{F}_S (G)$ is supersolvable, a contradiction. Hence our proof is complete.
\end{proof}
\begin{corollary}\label{generalised}
Let $G$ be a finite group and $S$ be a Sylow $p$-subgroup of $G$. Assume that there exists a subgroup $D$ of $S$ with $1<|D|<|S|$ such that every subgroup of $S$ of order $|D|$ or $p|D|$ is abelian, every subgroup of $S$ of order $|D|$ and every cyclic subgroup of $S$ of order 4 (if $S$ is a non-abelian 2-group and $|D|=2$) are generalized $S\Phi$-supplemented in $G$, then $\mathcal{F}_S (G)$ is supersolvable.
\end{corollary}
\begin{proof}
It follows from Theorem \ref{3.1} that we only need to prove that generalized $S\Phi$-supplemented property is a 1-eternal relation. Let $G$ be a finite group and $H$ be a subgroup of $G$ which is generalized $S\Phi$-supplemented in $G$. Then we obtain from Lemma {{\cite[Lemma 2.3 (1)]{MCG}}} that $H$ is generalized $S\Phi$-supplemented in $K$ for any $H \leq K \leq G$, which satisfies (0) in Definition \ref{eternal}. By Lemma \ref{30}, we conclude that generalized $S\Phi$-supplemented property  satisfies (1) in Definition \ref{eternal}. Therefore generalized $S\Phi$-supplemented property is an eternal relation and our proof is complete.
\end{proof}
\begin{corollary}\label{ICC}
Let $G$ be a finite group, $p$ be a prime divisor of $|G|$, and $S$ be a Sylow $p$-subgroup of $G$. Suppose that there exists a subgroup $1<D<S$ such that any subgroup of $S$ with order $|D|$ or $p|D|$ is abelian, any subgroup $H$ of $S$ of order $|D|$ or $4$ (If $S$ is a non-abelian $2$-group and $|D| =2$) is a strong $ICC$-subgroup of $G$, then $\mathcal{F}_S (G)$  is supersolvable.
\end{corollary}
\begin{proof}
It follows from Theorem \ref{3.1} that we only need to prove that generalized  strong $ICC$-property is a 1-eternal relation. Let $G$ be a finite group and $H$ be a subgroup of $G$ which is a strong $ICC$-subgroup of $G$. Then we obtain from Definition \ref{1.4} (1)  that $H$ is a strong $ICC$-subgroup of $K$ for any $H \leq K \leq G$, which satisfies (0) in Definition \ref{eternal}. By Lemma \ref{3}, we conclude that  strong $ICC$-property  satisfies (1) in Definition \ref{eternal}. Therefore strong $ICC$-property is an eternal relation and our proof is complete.
\end{proof}
\begin{theorem}\label{2-eternal}
Let $\rightsquigarrow$ be a 2-eternal relation, $G$ be a finite group, $p$ be a prime divisor of $|G|$, and $S$ be a Sylow $p$-subgroup of $G$. Suppose that there exists a subgroup $D$ of $S$ with $1 < |D| < |S|$ such that any subgroup of $S$ 
with order $|D|$ or $p|D|$ is abelian,  any subgroup $H$ of $S$ with order $|D|$ and any cyclic subgroup $H$ of $S$ with order $4$ (if $S$ is non-abelian, $|S| >4$ and  $p=2$) satisfies the relation $H \rightsquigarrow G$, then $\mathcal{F}_S (G)$ is supersolvable.
\end{theorem}
\begin{proof}
Assume that the theorem is false and let $G$ be a counterexample with $|G|$ minimal. Denote $\mathcal{F}_S (G)$ by $\mathcal{F}$.

Let $H$ be a proper subgroup of $G$ such that $S \cap H \in {\rm Syl}_p (H)$ and $|S \cap H| \geq p|D|$. We predicate that $\mathcal{F}_{S \cap H} (H)$ is supersolvable. By our hypothesis, every subgroup $T_1$ of $S \cap H$ with order $|D|$ and any cyclic subgroup $T_2$ of $S\cap H$ of order $4$ (if $S$ is non-abelian, $|S | >4$ and  $p=2$) satisfies the relation $T_1,T_2 \rightsquigarrow G$. Hence every subgroup $T_1$ of $S \cap H$ with order $|D|$ and any cyclic subgroup $T_2$ of $S\cap H$ of order $4$ (if $S \cap H$ is non-abelian, $|S \cap H| >4$ and  $p=2$) satisfies the relation $T_1,T_2 \rightsquigarrow H$. Since $1<|D|<|S \cap H|$, $H$ satisfies the hypothesis and it follows from the minimal choice of $G$ that $\mathcal{F}_{S \cap H} (H)$ is supersolvable. 
 Now let $Q \in \mathcal{E}_{\mathcal{F}} ^{*}$, Using the same method in Step 2,3 of Theorem \ref{3.1}, we conclude that $|Q| \geq p|D|$, $N_{\mathcal{F}} (Q)$ is supersolvable under the assumption that $Q \not\unlhd G$, and $|O_p (G)| \geq p |D|$.

Now we prove that $O_p (G) \leq Z_{\mathfrak{U}} (G)$. It follows from $|O_p (G)| \geq p |D|$ that there exists a subgroup $D_0 \leq O_p (G)$ with order $|D|$ such that every subgroup $T_1$ of $ O_p (G)$ with order $|D|$ and any cyclic subgroup $T_2$ of $ O_p (G)$ of order $4$ (if $O_p (G)$ is non-abelian, $|O_p (G) |>4$ and $p=2$) satisfies the relation $T_1,T_2 \rightsquigarrow G$. Since $\rightsquigarrow$ is a 2-eternal relation, it yields that $O_p (G) \leq Z_{\mathfrak{U}} (G)$ and this part is complete. Using the same method in Step 5 of Theorem \ref{3.1}, we can derive a contradiction, and the proof is complete.
\end{proof}
\begin{corollary}\label{U}
Let $G$ be a finite group, $p$ be a prime divisor of $|G|$, and $S$ be a Sylow $p$-subgroup of $G$. Suppose that there exists a subgroup $1<D<S$ such that any subgroup of $S$ with order $|D|$ or $p|D|$ is abelian, any subgroup of $S$ of order $|D|$ and any cyclic subgroup of $S$ of order $4$ (If $S$ is a non-abelian $2$-group, $|S|>4$ and $|D| =2$) is strongly $\mathfrak{U}$-embedded in $G$, then $\mathcal{F}_S (G)$ is supersolvable.
\end{corollary}
\begin{proof}
It follows from Theorem \ref{2-eternal} that we only need to prove that strongly $\mathfrak{U}$-embedded property is a 2-eternal relation. Let $G$ be a finite group and $H$ be a subgroup of $G$ which is strongly $\mathfrak{U}$-embedded in $G$. Then we obtain from Definition \ref{1.4} (2) that $H$ is strongly $\mathfrak{U}$-embedded in $K$ for any $H \leq K \leq G$, which satisfies (0) in Definition \ref{eternal}. By Lemma \ref{4}, we conclude that strongly $\mathfrak{U}$-embedded property  satisfies (2) in Definition \ref{eternal}. Therefore strongly $\mathfrak{U}$-embedded property is a 2-eternal relation and our proof is complete.
\end{proof}

\begin{theorem}\label{3-eternal}
Let $\rightsquigarrow$ be a 3-eternal relation, $G$ be a finite group, $p$ be a prime divisor of $|G|$, and $S$ be a Sylow $p$-subgroup of $G$. Suppose that there exists a subgroup $D$ of $S$ with $1 < |D| < |S|$ such that any subgroup of $S$ 
with order $|D|$ or $p|D|$ is abelian,  any subgroup $H$ of $S$ with order $|D|$ and any cyclic subgroup $H$ of $S$ with order $4$ (if $S$ is non-abelian and  $p=2$) satisfies the relation $H \rightsquigarrow G$, then $\mathcal{F}_S (G)$ is supersolvable.
\end{theorem}
\begin{proof}
Assume that the theorem is false and let $G$ be a counterexample with $|G|$ minimal. Denote $\mathcal{F}_S (G)$ by $\mathcal{F}$.

Let $H$ be a proper subgroup of $G$ such that $S \cap H \in {\rm Syl}_p (H)$ and $|S \cap H| \geq p|D|$. We predicate that $\mathcal{F}_{S \cap H} (H)$ is supersolvable. By our hypothesis, every subgroup $T_1$ of $S \cap H$ with order $|D|$ and any cyclic subgroup $T_2$ of $S\cap H$ of order $4$ (if $S$ is non-abelian  and  $p=2$) satisfies the relation $T_1,T_2 \rightsquigarrow G$. Hence every subgroup $T_1$ of $S \cap H$ with order $|D|$ and any cyclic subgroup $T_2$ of $S\cap H$ of order $4$ (if $S \cap H$ is non-abelian  and  $p=2$) satisfies the relation $T_1,T_2 \rightsquigarrow H$. Since $1<|D|<|S \cap H|$, $H$ satisfies the hypothesis and it follows from the minimal choice of $G$ that $\mathcal{F}_{S \cap H} (H)$ is supersolvable. 
 Now let $Q \in \mathcal{E}_{\mathcal{F}} ^{*}$, Using the same method in Step 2,3 of Theorem \ref{3.1}, we conclude that $|Q| \geq p|D|$, $N_{\mathcal{F}} (Q)$ is supersolvable under the assumption that $Q \not\unlhd G$, and $|O_p (G)| \geq p |D|$.

Now we prove that $O_p (G) \leq Z_{\mathfrak{U}} (G)$. It follows from $|O_p (G)| \geq p |D|$ that there exists a subgroup $D_0 \leq O_p (G)$ with order $|D|$ such that every subgroup $T_1$ of $ O_p (G)$ with order $|D|$ and any cyclic subgroup $T_2$ of $ O_p (G)$ of order $4$ (if $O_p (G)$ is non-abelian   and $p=2$) satisfies the relation $T_1,T_2 \rightsquigarrow G$. Since $\rightsquigarrow$ is a 3-eternal relation, it yields that $O_p (G) \leq Z_{\mathfrak{U}} (G)$ and this part is complete. Using the same method in Step 5 of Theorem \ref{3.1}, we can derive a contradiction, and the proof is complete.
\end{proof}

\begin{theorem}\label{4-eternal}
Let $\rightsquigarrow$ be a 4-eternal relation, $G$ be a finite group, $p$ be a prime divisor of $|G|$, and $S$ be a Sylow $p$-subgroup of $G$. Suppose that there exists a subgroup $D$ of $S$ with $1 < |D| < |S|$ such that any subgroup of $S$ 
with order $|D|$ or $p|D|$ is abelian,  any subgroup $H$ of $S$ with order $|D|$ and any cyclic subgroup $H$ of $S$ with order $4$ (if   $p=2$) satisfies the relation $H \rightsquigarrow G$, then $\mathcal{F}_S (G)$ is supersolvable.
\end{theorem}
\begin{proof}
Assume that the theorem is false and let $G$ be a counterexample with $|G|$ minimal. Denote $\mathcal{F}_S (G)$ by $\mathcal{F}$.

Let $H$ be a proper subgroup of $G$ such that $S \cap H \in {\rm Syl}_p (H)$ and $|S \cap H| \geq p|D|$. We predicate that $\mathcal{F}_{S \cap H} (H)$ is supersolvable. By our hypothesis, every subgroup $T_1$ of $S \cap H$ with order $|D|$ and any cyclic subgroup $T_2$ of $S\cap H$ of order $4$ (if   $p=2$) satisfies the relation $T_1,T_2 \rightsquigarrow G$. Hence every subgroup $T_1$ of $S \cap H$ with order $|D|$ and any cyclic subgroup $T_2$ of $S\cap H$ of order $4$ (if    $p=2$) satisfies the relation $T_1,T_2 \rightsquigarrow H$. Since $1<|D|<|S \cap H|$, $H$ satisfies the hypothesis and it follows from the minimal choice of $G$ that $\mathcal{F}_{S \cap H} (H)$ is supersolvable. 
 Now let $Q \in \mathcal{E}_{\mathcal{F}} ^{*}$, Using the same method in Step 2,3 of Theorem \ref{3.1}, we conclude that $|Q| \geq p|D|$, $N_{\mathcal{F}} (Q)$ is supersolvable under the assumption that $Q \not\unlhd G$, and $|O_p (G)| \geq p |D|$.

Now we prove that $O_p (G) \leq Z_{\mathfrak{U}} (G)$. It follows from $|O_p (G)| \geq p |D|$ that there exists a subgroup $D_0 \leq O_p (G)$ with order $|D|$ such that every subgroup $T_1$ of $ O_p (G)$ with order $|D|$ and any cyclic subgroup $T_2$ of $ O_p (G)$ of order $4$ (if   $p=2$) satisfies the relation $T_1,T_2 \rightsquigarrow G$. Since $\rightsquigarrow$ is a 4-eternal relation, it yields that $O_p (G) \leq Z_{\mathfrak{U}} (G)$ and this part is complete. Using the same method in Step 5 of Theorem \ref{3.1}, we can derive a contradiction, and the proof is complete.
\end{proof}

\begin{theorem}\label{5-eternal}
Let $\rightsquigarrow$ be a 5-eternal relation, $G$ be a finite group, $p$ be a prime divisor of $|G|$, and $S$ be a Sylow $p$-subgroup of $G$. Suppose that there exists a subgroup $D$ of $S$ with $1 \leq |D| < |S|$ such that any subgroup of $S$ 
with order $|D|$ or $p|D|$ is abelian and satisfies the relation $H \rightsquigarrow G$, then $\mathcal{F}_S (G)$ is supersolvable.
\end{theorem}
\begin{proof}
Assume that the theorem is false and let $G$ be a counterexample with $|G|$ minimal. Denote $\mathcal{F}_S (G)$ by $\mathcal{F}$.

Let $H$ be a proper subgroup of $G$ such that $S \cap H \in {\rm Syl}_p (H)$ and $|S \cap H| \geq p|D|$. We predicate that $\mathcal{F}_{S \cap H} (H)$ is supersolvable. By our hypothesis, every subgroup $T $ of $S \cap H$ with order $|D|$ or $p|D|$   satisfies the relation $T  \rightsquigarrow G$. Hence every subgroup $T $ of $S \cap H$ with order $|D|$ or $p|D|$  satisfies the relation $T  \rightsquigarrow H$. Since $1\leq |D|<|S \cap H|$, $H$ satisfies the hypothesis and it follows from the minimal choice of $G$ that $\mathcal{F}_{S \cap H} (H)$ is supersolvable. 
 Now let $Q \in \mathcal{E}_{\mathcal{F}} ^{*}$, Using the same method in Step 2,3 of Theorem \ref{3.1}, we conclude that $|Q| \geq p|D|$, $N_{\mathcal{F}} (Q)$ is supersolvable under the assumption that $Q \not\unlhd G$, and $|O_p (G)| \geq p |D|$.

Now we prove that $O_p (G) \leq Z_{\mathfrak{U}} (G)$. It follows from $|O_p (G)| \geq p |D|$ that there exists a subgroup $D_0 \leq O_p (G)$ with order $|D|$ such that every subgroup $T $ of $ O_p (G)$ with order $|D|$ or $p|D|$  satisfies the relation $T  \rightsquigarrow G$. Since $\rightsquigarrow$ is a 5-eternal relation, it yields that $O_p (G) \leq Z_{\mathfrak{U}} (G)$ and this part is complete. Using the same method in Step 5 of Theorem \ref{3.1}, we can derive a contradiction, and the proof is complete.
\end{proof}
\begin{corollary}\label{c-supplemented}
Let $G$ be a finite group, $p$ be an odd prime divisor of $|G|$, and $S$ be a Sylow $p$-subgroup of $G$. Suppose that $S$ has a subgroup $D$ such that $1 \leq D <S$, and all subgroups of $S$ of order $|D|$ or $p|D|$ are abelian and $c$-supplemented in $G$, then $\mathcal{F}_S (G)$ is supersolvable.
\end{corollary}
\begin{proof}
It follows from Theorem \ref{5-eternal} that we only need to prove that  $c$-supplemented property is a 5-eternal relation. Let $G$ be a finite group and $H$ be a subgroup of $G$ which is $c$-supplemented in $G$. Then we obtain from Lemma \ref{21} that $H$ is  $c$-supplemented in  $K$ for any $H \leq K \leq G$, which satisfies (0) in Definition \ref{eternal}. By Lemma \ref{22}, we conclude that  $c$-supplemented  property  satisfies (5) in Definition \ref{eternal}. Therefore  $c$-supplemented   property   is a 5-eternal relation and our proof is complete.
\end{proof}

\begin{theorem}\label{6-eternal}
Let $\rightsquigarrow$ be a 6-eternal relation, $G$ be a finite group, $p$ be a prime divisor of $|G|$, and $S$ be a Sylow $p$-subgroup of $G$. Suppose that there exists a subgroup $D$ of $S$ with $1 \leq |D| < |S|$ such that any subgroup of $S$ 
with order $|D|$ or $p|D|$ is abelian,  any subgroup of $S$ 
with order $|D|$   satisfies the relation $H \rightsquigarrow G$, then $\mathcal{F}_S (G)$ is supersolvable.
\end{theorem}
\begin{proof}
Assume that the theorem is false and let $G$ be a counterexample with $|G|$ minimal. Denote $\mathcal{F}_S (G)$ by $\mathcal{F}$.

Let $H$ be a proper subgroup of $G$ such that $S \cap H \in {\rm Syl}_p (H)$ and $|S \cap H| \geq p|D|$. We predicate that $\mathcal{F}_{S \cap H} (H)$ is supersolvable. By our hypothesis, every subgroup $T $ of $S \cap H$ with order $|D|$  satisfies the relation $T  \rightsquigarrow G$. Hence every subgroup $T $ of $S \cap H$ with order $|D|$    satisfies the relation $T  \rightsquigarrow H$. Since $1\leq |D|<|S \cap H|$, $H$ satisfies the hypothesis and it follows from the minimal choice of $G$ that $\mathcal{F}_{S \cap H} (H)$ is supersolvable. 
 Now let $Q \in \mathcal{E}_{\mathcal{F}} ^{*}$, Using the same method in Step 2,3 of Theorem \ref{3.1}, we conclude that $|Q| \geq p|D|$, $N_{\mathcal{F}} (Q)$ is supersolvable under the assumption that $Q \not\unlhd G$, and $|O_p (G)| \geq p |D|$.

Now we prove that $O_p (G) \leq Z_{\mathfrak{U}} (G)$. It follows from $|O_p (G)| \geq p |D|$ that there exists a subgroup $D_0 \leq O_p (G)$ with order $|D|$ such that every subgroup $T $ of $ O_p (G)$ with order $|D|$    satisfies the relation $T  \rightsquigarrow G$. Since $\rightsquigarrow$ is a 6-eternal relation, it yields that $O_p (G) \leq Z_{\mathfrak{U}} (G)$ and this part is complete. Using the same method in Step 5 of Theorem \ref{3.1}, we can derive a contradiction, and the proof is complete.
\end{proof}
\begin{corollary}\label{weakly-H-2}
Let $G$ be a finite group, $p$ be a prime divisor of $|G|$, and $S$ be a Sylow $p$-subgroup of $G$. Suppose that $p$ is odd, and there exists a subgroup $D$ of $S$ such that every subgroup of $S$ of order $|D|$ is a weakly $\mathcal{H}$-subgroup of $G$, every subgroup of $S$ of order $|D|$ or $p|D|$ is abelian. Then $\mathcal{F}_S (G)$ is supersolvable.
\end{corollary}
\begin{proof}
It follows from Theorem \ref{6-eternal} that we only need to prove that   weakly $\mathcal{H}$-property is a 6-eternal relation. Let $G$ be a finite group and $H$ be a subgroup of $G$ which is  weakly $\mathcal{H}$ in $G$. Then we obtain from Lemma \ref{23} that $H$ is  a  weakly $\mathcal{H}$-subgroup of  $K$ for any $H \leq K \leq G$, which satisfies (0) in Definition \ref{eternal}. By Lemma \ref{24} (2), we conclude that  weakly $\mathcal{H}$-property  satisfies (6) in Definition \ref{eternal}. Therefore  weakly $\mathcal{H}$-property   is a 6-eternal relation and our proof is complete.
\end{proof}

\begin{theorem}\label{7-eternal}
Let $\rightsquigarrow$ be a 7-eternal relation, $G$ be a finite group, $p$ be a prime divisor of $|G|$, and $S$ be a Sylow $p$-subgroup of $G$. Suppose that there exists a subgroup $D$ of $S$ with $1 \leq |D| < |S|$ such that any subgroup of $S$ 
with order $|D|$ or $p|D|$ is abelian,  any subgroup of $S$ 
with order $|D|$   satisfies the relation $H \rightsquigarrow G$, and any subgroup of $S$ 
with order $4$ (if $|D| = 2 = p$)  satisfies the relation $H \rightsquigarrow G$,  then $\mathcal{F}_S (G)$ is supersolvable.
\end{theorem}
\begin{proof}
Assume that the theorem is false and let $G$ be a counterexample with $|G|$ minimal. Denote $\mathcal{F}_S (G)$ by $\mathcal{F}$.

Let $H$ be a proper subgroup of $G$ such that $S \cap H \in {\rm Syl}_p (H)$ and $|S \cap H| \geq p|D|$. We predicate that $\mathcal{F}_{S \cap H} (H)$ is supersolvable. By our hypothesis, every subgroup $T_1 $ of $S \cap H$ with order $|D|$ and every subgroup $T_2$ of $S$ 
with order $4$ (if $|D| = 2 = p$)   satisfies the relation $T_1, T_2  \rightsquigarrow G$. Hence every subgroup $T_1$ of $S \cap H$ with order $|D|$  and every subgroup $T_2$ of $S$ 
with order $4$ (if $|D| = 2 = p$)   satisfies the relation $T_1,T_2   \rightsquigarrow H$. Since $1\leq |D|<|S \cap H|$, $H$ satisfies the hypothesis and it follows from the minimal choice of $G$ that $\mathcal{F}_{S \cap H} (H)$ is supersolvable. 
 Now let $Q \in \mathcal{E}_{\mathcal{F}} ^{*}$, Using the same method in Step 2,3 of Theorem \ref{3.1}, we conclude that $|Q| \geq p|D|$, $N_{\mathcal{F}} (Q)$ is supersolvable under the assumption that $Q \not\unlhd G$, and $|O_p (G)| \geq p |D|$.

Now we prove that $O_p (G) \leq Z_{\mathfrak{U}} (G)$. It follows from $|O_p (G)| \geq p |D|$ that there exists a subgroup $D_0 \leq O_p (G)$ with order $|D|$ such that every subgroup $T_1 $ of $ O_p (G)$ with order $|D|$  and every subgroup $T_2$ of $O_p (G)$ 
with order $4$ (if $|D| = 2 = p$)   satisfies the relation $T_1, T_2  \rightsquigarrow G$. Since $\rightsquigarrow$ is a 7-eternal relation, it yields that $O_p (G) \leq Z_{\mathfrak{U}} (G)$ and this part is complete. Using the same method in Step 5 of Theorem \ref{3.1}, we can derive a contradiction, and the proof is complete.
\end{proof}
\begin{corollary}\label{HC}
Let $G$ be a finite group, $p$ be a prime divisor of $|G|$, and $S$ be a Sylow $p$-subgroup of $G$. Suppose that $S$ has a subgroup $D$ such that $1 < |D |<|S|$, and all subgroups of $S$ of order $|D|$ or $p|D|$ are abelian, every subgroup of order $|D|$ or $4$ (if $p = |D|=2$) is an $\mathscr{H} C$-subgroup of $G$.  Then $\mathcal{F}_S (G)$ is supersolvable. 
\end{corollary}
\begin{proof}
It follows from Theorem \ref{7-eternal} that we only need to prove that   $\mathscr{H} C$-property is a 7-eternal relation. Let $G$ be a finite group and $H$ be a subgroup of $G$ which    is an $\mathscr{H} C$-subgroup of  $G$. Then we obtain from {{\cite[Lemma 2.3 (1)]{GX}}} that $H$ is an $\mathscr{H} C$-subgroup of $K$ for any $H \leq K \leq G$, which satisfies (0) in Definition \ref{eternal}. By Lemma \ref{27}, we conclude that  $\mathscr{H} C$-property  satisfies (7) in Definition \ref{eternal}. Therefore  $\mathscr{H} C$-property   is a 7-eternal relation and our proof is complete.
\end{proof}

\begin{theorem}\label{8-eternal}
Let $\rightsquigarrow$ be a 8-eternal relation, $G$ be a finite group, $p$ be a prime divisor of $|G|$, and $S$ be a Sylow $p$-subgroup of $G$. Suppose that there exists a subgroup $D$ of $S$ with $1 < |D| < |S|$ such that any subgroup of $S$ 
with order $|D|$ or $p|D|$ is abelian,  any subgroup of $S$ 
with order $|D|$   satisfies the relation $H \rightsquigarrow G$, and any subgroup of $S$ 
with order $2|D|$ (if $S$ is non-abelian, $p=2$ and $|S:D|>2$)   satisfies the relation $H \rightsquigarrow G$,  then $\mathcal{F}_S (G)$ is supersolvable.
\end{theorem}
\begin{proof}
Assume that the theorem is false and let $G$ be a counterexample with $|G|$ minimal. Denote $\mathcal{F}_S (G)$ by $\mathcal{F}$.

Let $H$ be a proper subgroup of $G$ such that $S \cap H \in {\rm Syl}_p (H)$ and $|S \cap H| \geq p|D|$. We predicate that $\mathcal{F}_{S \cap H} (H)$ is supersolvable. By our hypothesis, every subgroup $T_1 $ of $S \cap H$ with order $|D|$ and every subgroup $T_2$ of $S$ 
with order $2|D|$ (if $S \cap H$ is non-abelian, $p=2$ and $|S\cap H :D|>2$)    satisfies the relation $T_1, T_2  \rightsquigarrow G$. Hence every subgroup $T_1$ of $S \cap H$ with order $|D|$  and every subgroup $T_2$ of $S$ 
with order $2|D|$ (if $S \cap H$ is non-abelian, $p=2$ and $|S \cap H :D|>2$)   satisfies the relation $T_1,T_2   \rightsquigarrow H$. Since $1< |D|<|S \cap H|$, $H$ satisfies the hypothesis and it follows from the minimal choice of $G$ that $\mathcal{F}_{S \cap H} (H)$ is supersolvable. 
 Now let $Q \in \mathcal{E}_{\mathcal{F}} ^{*}$, Using the same method in Step 2,3 of Theorem \ref{3.1}, we conclude that $|Q| \geq p|D|$, $N_{\mathcal{F}} (Q)$ is supersolvable under the assumption that $Q \not\unlhd G$, and $|O_p (G)| \geq p |D|$.

Now we prove that $O_p (G) \leq Z_{\mathfrak{U}} (G)$. It follows from $|O_p (G)| \geq p |D|$ that there exists a subgroup $D_0 \leq O_p (G)$ with order $|D|$ such that every subgroup $T_1 $ of $ O_p (G)$ with order $|D|$  and every subgroup $T_2$ of $O_p (G)$ 
with order $2 |D|$ (if $O_p (G)$ is non-abelian, $p=2$ and $|O_p (G):D|>2$)  satisfies the relation $T_1, T_2  \rightsquigarrow G$. Since $\rightsquigarrow$ is a 8-eternal relation, it yields that $O_p (G) \leq Z_{\mathfrak{U}} (G)$ and this part is complete. Using the same method in Step 5 of Theorem \ref{3.1}, we can derive a contradiction, and the proof is complete.
\end{proof}
\begin{corollary}\label{SS}
Let $G$ be a finite group, $p$ be a prime divisor of $|G|$ and $S$ be a Sylow $p$-subgroup of $G$. Suppose that there exists a subgroup $D$ of $S$ such that $1<|D|<|S|$, every subgroup of $S$ of order $|D|$ or $p|D|$ is abelian, every subgroup of $S$ of order $|D|$ or $2|D|$ (if $S$ is non-abelian, $p=2$ and $|S:D|>2$) is either $S$-quasinormally embedded or $SS$-quasinormal in $G$, then $\mathcal{F}_S (G)$ is supersolvable. 
\end{corollary}
\begin{proof}
It follows from Theorem \ref{8-eternal} that we only need to prove that   either $S$-quasinormally embedded or $SS$-quasinormal property is a 8-eternal relation. Let $G$ be a finite group and $H$ be a subgroup of $G$ which    is either $S$-quasinormally embedded or $SS$-quasinormal in  $G$. Then we obtain from {{\cite[Lemma 2.1 (i)]{LSK}}} and {{\cite[Lemma 2.1 (a)]{XX}}} that $H$ is either $S$-quasinormally embedded or $SS$-quasinormal in  $K$ for any $H \leq K \leq G$, which satisfies (0) in Definition \ref{eternal}. By Lemma \ref{28}, we conclude that  $\mathscr{H} C$-property  satisfies (8) in Definition \ref{eternal}. Therefore  $\mathscr{H} C$-property   is a 8-eternal relation and our proof is complete.
\end{proof}
\begin{theorem}\label{9-eternal}
Let $\rightsquigarrow$ be a 9-eternal relation, $G$ be a finite group, $p$ be a prime divisor of $|G|$, and $S$ be a Sylow $p$-subgroup of $G$. Suppose that there exists a subgroup $D$ of $S$ with $1 < |D| < |S|$ such that any subgroup of $S$ 
with order $|D|$ or $p|D|$ is abelian,  any subgroup of $S$ 
with order $|D|$ and $p|D|$   satisfies the relation $H \rightsquigarrow G$,  then $\mathcal{F}_S (G)$ is supersolvable.
\end{theorem}
\begin{proof}
Assume that the theorem is false and let $G$ be a counterexample with $|G|$ minimal. Denote $\mathcal{F}_S (G)$ by $\mathcal{F}$.

Let $H$ be a proper subgroup of $G$ such that $S \cap H \in {\rm Syl}_p (H)$ and $|S \cap H| \geq p|D|$. We predicate that $\mathcal{F}_{S \cap H} (H)$ is supersolvable. By our hypothesis, every subgroup $T_1 $ of $S \cap H$ with order $|D|$ and every subgroup $T_2$ of $S$ 
with order $p|D|$ satisfies the relation $T_1, T_2  \rightsquigarrow G$. Hence every subgroup $T_1$ of $S \cap H$ with order $|D|$  and every subgroup $T_2$ of $S$ 
with order $p|D|$    satisfies the relation $T_1,T_2   \rightsquigarrow H$. Since $1< |D|<|S \cap H|$, $H$ satisfies the hypothesis and it follows from the minimal choice of $G$ that $\mathcal{F}_{S \cap H} (H)$ is supersolvable. 
 Now let $Q \in \mathcal{E}_{\mathcal{F}} ^{*}$, Using the same method in Step 2,3 of Theorem \ref{3.1}, we conclude that $|Q| \geq p|D|$, $N_{\mathcal{F}} (Q)$ is supersolvable under the assumption that $Q \not\unlhd G$, and $|O_p (G)| \geq p |D|$.

Now we prove that $O_p (G) \leq Z_{\mathfrak{U}} (G)$. It follows from $|O_p (G)| \geq p |D|$ that there exists a subgroup $D_0 \leq O_p (G)$ with order $|D|$ such that every subgroup $T_1 $ of $ O_p (G)$ with order $|D|$  and every subgroup $T_2$ of $O_p (G)$ 
with order $p|D|$    satisfies the relation $T_1, T_2  \rightsquigarrow G$. Since $\rightsquigarrow$ is a 9-eternal relation, it yields that $O_p (G) \leq Z_{\mathfrak{U}} (G)$ and this part is complete. Using the same method in Step 5 of Theorem \ref{3.1}, we can derive a contradiction, and the proof is complete.
\end{proof}
\begin{corollary}\label{weakly HC-embedded}
Let $G$ be a finite group, $p$ be a prime divisor of $|G|$ and $S$ be a Sylow $p$-subgroup of $G$. Suppose that there exists a subgroup $D$ of $S$ such that $1<|D|<|S|$, every subgroup of $S$ of order $|D|$ or $p|D|$ is abelian and weakly $\mathcal{HC}$-embedded in $G$, then $\mathcal{F}_S (G)$ is supersolvable. 
\end{corollary}
\begin{proof}
It follows from Theorem \ref{9-eternal} that we only need to prove that  weakly $\mathcal{HC}$-embedded inproperty is a 9-eternal relation. Let $G$ be a finite group and $H$ be a subgroup of $G$ which    is weakly $\mathcal{HC}$-embedded   in  $G$. Then we obtain from{{\cite[Lemma 2.1 (1)]{HX}}} that $H$ is weakly $\mathcal{HC}$-embedded   in  $K$ for any $H \leq K \leq G$, which satisfies (0) in Definition \ref{eternal}. By Lemma \ref{29}, we conclude that  weakly $\mathcal{HC}$-embedded   property  satisfies (9) in Definition \ref{eternal}. Therefore  weakly $\mathcal{HC}$-embedded  property   is a 9-eternal relation and our proof is complete.
\end{proof}

\begin{theorem}\label{semi-eternal}
Let $\rightsquigarrow$ be a 1-semi-eternal relation, $G$ be a finite group, $p$ be a prime divisor of $|G|$, and $S$ be a Sylow $p$-subgroup of $G$. Suppose that any subgroup $H$ of $S$ with order $p$ or any cyclic subgroup $H$ of $S$ with order $4$ (if $S$ is non-abelian and $p=2$) satisfies the relation $H \rightsquigarrow G$, then $\mathcal{F}_S (G)$ is supersolvable.
\end{theorem}
\begin{proof}
Assume that the theorem is false and let $G$ be a counterexample with $|G|$ minimal. Denote $\mathcal{F}_S (G)$ by $\mathcal{F}$.

Let $H$ be a proper subgroup of $G$ such that $S \cap H \in {\rm Syl}_p (H)$ and $|S \cap H| \geq p^2$. We predicate that $\mathcal{F}_{S \cap H} (H)$ is supersolvable. By our hypothesis, every subgroup $T_1$ of $S \cap H$ with order $p$ and any cyclic subgroup $T_2$ of $S\cap H$ of order $4$ (If $S \cap H$ is non-abelian and $p=2$) satisfies the relation $T_1,T_2 \rightsquigarrow G$. Hence every subgroup $T_1$ of $S \cap H$ with order $p$ and any cyclic subgroup $T_2$ of $S\cap H$ of order $4$ (if $S \cap H$ is non-abelian and $p=2$) satisfies the relation $T_1,T_2 \rightsquigarrow H$. Since $1<p<|S \cap H|$, $H$ satisfies the hypothesis and it follows from the minimal choice of $G$ that $\mathcal{F}_{S \cap H} (H)$ is supersolvable. 
 Now let $Q \in \mathcal{E}_{\mathcal{F}} ^{*}$, Using the same method in Step 2,3 of Theorem \ref{3.1}, we conclude that $|Q| \geq p^2$, $N_{\mathcal{F}} (Q)$ is supersolvable under the assumption that $Q \not\unlhd G$, and $|O_p (G)| \geq p^2$.

Now we prove that $O_p (G) \leq Z_{\mathfrak{U}} (G)$. It follows from $|O_p (G)| \geq p^2$ that there exists a subgroup $D_0 \leq O_p (G)$ with order $p$ such that every subgroup $T_1$ of $ O_p (G)$ with order $p$ and any cyclic subgroup $T_2$ of $ O_p (G)$ of order $4$ (If $O_p (G)$ is non-abelian and $p=2$) satisfies the relation $T_1,T_2 \rightsquigarrow G$. Since $\rightsquigarrow$ is a 1-semi-eternal relation, it yields that $O_p (G) \leq Z_{\mathfrak{U}} (G)$ and this part is complete. Using the same method in Step 5 of Theorem \ref{3.1}, we can derive a contradiction, and the proof is complete.
\end{proof}
\begin{corollary}\label{ICSC}
Let $G$ be a finite group, $p$ be a prime divisor of $|G|$, and $S$ be a Sylow $p$-subgroup of $G$. Suppose that any subgroup $H$ of $S$ with order $p$ or any cyclic subgroup $H$ of $S$ with order $4$ (if $S$ is non-abelian and $p=2$)  is an $ICSC$-subgroup of $G$, then $\mathcal{F}_S (G)$ is supersolvable.
\end{corollary}
\begin{proof}
It follows from Theorem \ref{semi-eternal} that we only need to prove that $ICSC$-property is a semi-eternal relation. Let $G$ be a finite group and $H$ be a subgroup of $G$ which is an $ICSC$-subgroup of $G$. Then we obtain from Lemma \ref{18} that $H$ is an $ICSC$-subgroup of $K$ for any $H \leq K \leq G$, which satisfies (0) in Definition \ref{eternal}. By Lemma \ref{19}, we conclude that $ICSC$-property  satisfies (1') in Definition \ref{eternal}. Therefore $ICSC$-property is a 1-semi-eternal relation and our proof is complete.
\end{proof} 
\begin{corollary}\label{E-supplemented}
Let $G$ be a finite group, $p$ be a prime divisor of $|G|$, and $S$ be a Sylow $p$-subgroup of $G$. Suppose that any subgroup $H$ of $S$ with order $p$ or any cyclic subgroup $H$ of $S$ with order $4$ (if $S$ is non-abelian and $p=2$)  is $E$-supplemented in $G$, then $\mathcal{F}_S (G)$ is supersolvable.
\end{corollary}
\begin{proof}
It follows from Theorem \ref{semi-eternal} that we only need to prove that  $E$-supplemented property is a 1-semi-eternal relation. Let $G$ be a finite group and $H$ be a subgroup of $G$ which is $E$-supplemented in $G$. Then we obtain from {{\cite[Lemma 2.3 (1)]{LI}}} that $H$ is $E$-supplemented in $K$ for any $H \leq K \leq G$, which satisfies (0) in Definition \ref{eternal}. By Lemma \ref{17}, we conclude that $E$-supplemented   property  satisfies (1') in Definition \ref{eternal}. Therefore $E$-supplemented property   is a 1-semi-eternal relation and our proof is complete.
\end{proof}
 
 \begin{theorem}\label{2-semi-eternal}
Let $\rightsquigarrow$ be a 2-semi-eternal relation, $G$ be a finite group, $p$ be a prime divisor of $|G|$, and $S$ be a Sylow $p$-subgroup of $G$. Suppose that any subgroup $H$ of $S$ with order $p$ or any cyclic subgroup $H$ of $S$ with order $4$ (if $p=2$) satisfies the relation $H \rightsquigarrow G$, then $\mathcal{F}_S (G)$ is supersolvable.
\end{theorem}
\begin{proof}
Assume that the theorem is false and let $G$ be a counterexample with $|G|$ minimal. Denote $\mathcal{F}_S (G)$ by $\mathcal{F}$.

Let $H$ be a proper subgroup of $G$ such that $S \cap H \in {\rm Syl}_p (H)$ and $|S \cap H| \geq p^2$. We predicate that $\mathcal{F}_{S \cap H} (H)$ is supersolvable. By our hypothesis, every subgroup $T_1$ of $S \cap H$ with order $p$ and any cyclic subgroup $T_2$ of $S\cap H$ of order $4$ (if $p=2$) satisfies the relation $T_1,T_2 \rightsquigarrow G$. Hence every subgroup $T_1$ of $S \cap H$ with order $p$ and any cyclic subgroup $T_2$ of $S\cap H$ of order $4$ (if $p=2$) satisfies the relation $T_1,T_2 \rightsquigarrow H$. Since $1<p<|S \cap H|$, $H$ satisfies the hypothesis and it follows from the minimal choice of $G$ that $\mathcal{F}_{S \cap H} (H)$ is supersolvable. 
 Now let $Q \in \mathcal{E}_{\mathcal{F}} ^{*}$, Using the same method in Step 2,3 of Theorem \ref{3.1}, we conclude that $|Q| \geq p^2$, $N_{\mathcal{F}} (Q)$ is supersolvable under the assumption that $Q \not\unlhd G$, and $|O_p (G)| \geq p^2$.

Now we prove that $O_p (G) \leq Z_{\mathfrak{U}} (G)$. It follows from $|O_p (G)| \geq p^2$ that there exists a subgroup $D_0 \leq O_p (G)$ with order $p$ such that every subgroup $T_1$ of $ O_p (G)$ with order $p$ and any cyclic subgroup $T_2$ of $ O_p (G)$ of order $4$ (if $p=2$) satisfies the relation $T_1,T_2 \rightsquigarrow G$. Since $\rightsquigarrow$ is a 2-semi-eternal relation, it yields that $O_p (G) \leq Z_{\mathfrak{U}} (G)$ and this part is complete. Using the same method in Step 5 of Theorem \ref{3.1}, we can derive a contradiction, and the proof is complete.
\end{proof}

\begin{corollary}\label{pCAP}
Let $G$ be a finite group, $p$ be a prime divisor of $|G|$, and $S$ be a Sylow $p$-subgroup of $G$. Suppose that any subgroup $H$ of $S$ with order $p$ or any cyclic subgroup $H$ of $S$ with order $4$ (if $p=2$)  is a partial $CAP$-subgroup of $G$, then $\mathcal{F}_S (G)$ is supersolvable.
\end{corollary}
\begin{proof}
It follows from Theorem \ref{2-semi-eternal} that we only need to prove that  partial $CAP$-property is a 4-eternal relation. Let $G$ be a finite group and $H$ be a subgroup of $G$ which is a partial $CAP$-subgroup of $G$. Then we obtain from {{\cite[Lemma 2.1]{BB}}} that $H$ is a partial $CAP$-subgroup of $K$ for any $H \leq K \leq G$, which satisfies (0) in Definition \ref{eternal}. By Lemma \ref{20}, we conclude that  partial $CAP$-property  satisfies (4) in Definition \ref{eternal}. Therefore  partial $CAP$-property   is a 4-eternal relation and our proof is complete.
\end{proof}

\begin{theorem}\label{3-semi-eternal}
Let $\rightsquigarrow$ be a 3-semi-eternal relation, $G$ be a finite group, $p$ be a prime divisor of $|G|$, and $S$ be a Sylow $p$-subgroup of $G$. Suppose that any subgroup $H$ of $S$ with order $p$ or any cyclic subgroup $H$ of $S$ with order $4$ (if $S$ is non-cyclic 2-group) satisfies the relation $H \rightsquigarrow G$, then $\mathcal{F}_S (G)$ is supersolvable.
\end{theorem}
\begin{proof}
Assume that the theorem is false and let $G$ be a counterexample with $|G|$ minimal. Denote $\mathcal{F}_S (G)$ by $\mathcal{F}$.

Let $H$ be a proper subgroup of $G$ such that $S \cap H \in {\rm Syl}_p (H)$ and $|S \cap H| \geq p^2$. We predicate that $\mathcal{F}_{S \cap H} (H)$ is supersolvable. By our hypothesis, every subgroup $T_1$ of $S \cap H$ with order $p$ and any cyclic subgroup $T_2$ of $S\cap H$ of order $4$ (if $S \cap H$ is non-cyclic 2-group) satisfies the relation $T_1,T_2 \rightsquigarrow G$. Hence every subgroup $T_1$ of $S \cap H$ with order $p$ and any cyclic subgroup $T_2$ of $S\cap H$ of order $4$ (if $S \cap H$ is non-cyclic 2-group) satisfies the relation $T_1,T_2 \rightsquigarrow H$. Since $1<p<|S \cap H|$, $H$ satisfies the hypothesis and it follows from the minimal choice of $G$ that $\mathcal{F}_{S \cap H} (H)$ is supersolvable. 
 Now let $Q \in \mathcal{E}_{\mathcal{F}} ^{*}$, Using the same method in Step 2,3 of Theorem \ref{3.1}, we conclude that $|Q| \geq p^2$, $N_{\mathcal{F}} (Q)$ is supersolvable under the assumption that $Q \not\unlhd G$, and $|O_p (G)| \geq p^2$.

Now we prove that $O_p (G) \leq Z_{\mathfrak{U}} (G)$. It follows from $|O_p (G)| \geq p^2$ that there exists a subgroup $D_0 \leq O_p (G)$ with order $p$ such that every subgroup $T_1$ of $ O_p (G)$ with order $p$ and any cyclic subgroup $T_2$ of $ O_p (G)$ of order $4$ (if $O_p (G)$ is non-cyclic 2-group) satisfies the relation $T_1,T_2 \rightsquigarrow G$. Since $\rightsquigarrow$ is a 3-semi-eternal relation, it yields that $O_p (G) \leq Z_{\mathfrak{U}} (G)$ and this part is complete. Using the same method in Step 5 of Theorem \ref{3.1}, we can derive a contradiction, and the proof is complete.
\end{proof}
\begin{corollary}\label{weakly-H-1}
Let $G$ be a finite group, $p$ be a prime divisor of $|G|$, and $S$ be a Sylow $p$-subgroup of $G$. Suppose that every cyclic subgroup of $S$ of order $p$ or $4$ (if $P$ is a non-cyclic $2$-group)  is a weakly $\mathcal{H}$-subgroup of $G$, then $\mathcal{F}_S (G)$ is supersolvable.
\end{corollary}
\begin{proof}
It follows from Theorem \ref{3-semi-eternal} that we only need to prove that weakly $\mathcal{H}$-property is a 3-semi-eternal relation. Let $G$ be a finite group and $H$ be a subgroup of $G$ which is a  weakly $\mathcal{H}$-subgroup of $G$. Then we obtain from Lemma \ref{23} that $H$ is an weakly $\mathcal{H}$-subgroup of $K$ for any $H \leq K \leq G$, which satisfies (0) in Definition \ref{eternal}. By Lemma \ref{24} (1), we conclude that weakly $\mathcal{H}$-property  satisfies (3') in Definition \ref{eternal}. Therefore weakly $\mathcal{H}$-property is a 3-semi-eternal relation and our proof is complete.
\end{proof} 
Now,our generalizations seem to be less powerful when we encounter with the following generalized normalities. However, the main part of the proof is still the same, just some slight differences need to be changed.

\begin{theorem}\label{weakly SPhi}
Let $G$ be a finite group, $p$ be a prime divisor of $|G|$, and $S$ be a Sylow $p$-subgroup of $G$. Suppose that $S$ has a subgroup $D$ such that $1 \leq D <S$, and all subgroups of $S$ of order $|D|$ or $p|D|$ are abelian and weakly $S\Phi$-supplemented in $G$. If $S$ is a non-abelian $2$-group and $|D|=1$, we have additionally that every cyclic subgroup of $G$ of order $4$ is a weakly $S\Phi$-supplemented subgroup of $G$.  Then $\mathcal{F}_S (G)$ is supersolvable. 
\end{theorem}
\begin{proof}
Assume that the theorem is false and let $G$ be a counterexample with $|G|$ minimal. Denote $\mathcal{F}_S (G)$ by $\mathcal{F}$.

Let $H$ be a proper subgroup of $G$ such that $S \cap H \in {\rm Syl}_p (H)$ and $|S \cap H| \geq p |D|$. We predicate that $\mathcal{F}_{S \cap H} (H)$ is supersolvable. By our hypothesis, every subgroup of $S \cap H$ with order $|D|$ or $p|D|$ is weakly $S\Phi$-supplemented in $G$. If $S \cap H$ is a non-abelian $2$-group and $|D|=1$, then every cyclic subgroup of $S \cap H$ of order $4$ is a weakly $S\Phi$-supplemented subgroup of $G$. Hence every subgroup of $S \cap H$ with order $|D|$ or $p|D|$ is weakly $S\Phi$-supplemented in $H$, and every cyclic subgroup of $G$ of order $4$ is a weakly $S\Phi$-supplemented subgroup of $H$ if $S \cap H$ is a non-abelian $2$-group and $|D|=1$ by Lemma \ref{25}. Clearly any subgroup of $S \cap H$ with order $|D|$ or $p|D|$ is abelian. Since $1\leq |D| <|S \cap H|$, $H$ satisfies the hypothesis and it follows from the minimal choice of $G$ that $\mathcal{F}_{S \cap H} (H)$ is supersolvable. 
 Now let $Q \in \mathcal{E}_{\mathcal{F}} ^{*}$, Using the same method in Step 2,3 of Theorem \ref{3.1}, we conclude that $|Q| \geq p|D|$, $N_{\mathcal{F}} (Q)$ is supersolvable under the assumption that $Q \not\unlhd G$, and $|O_p (G)| \geq p|D|$.

Now we prove that $O_p (G) \leq Z_{\mathfrak{U}} (G)$. It follows from $|O_p (G)| \geq p|D|$ that there exists a subgroup $D_0 \leq O_p (G)$ with order $|D|$ such that any subgroup of $O_p (G)$ of order $|D|$ or $p|D|$ is abelian and weakly $S\Phi$-supplemented in $G$. If $O_p (G)$ is a non-abelian $2$-group and $|D|=1$, then every cyclic subgroup of $O_p (G)$ of order $4$ is a weakly $S\Phi$-supplemented subgroup of $G$. By Lemma \ref{26}, it yields that $O_p (G) \leq Z_{\mathfrak{U}} (G)$ and this part is complete. Using the same method in Step 5 of Theorem \ref{3.1}, we can derive a contradiction, and the proof is complete.
\end{proof}
\begin{theorem}\label{NE*}
Let $G$ be a $Q_8$-$free$ finite group, $p$ be a prime divisor of $|G|$, and $S$ be a Sylow $p$-subgroup of $G$. Suppose that any subgroup $H$ of $S$ with order $p$ is an $NE^{*}$-subgroup of $G$, then $\mathcal{F}_S (G)$ is supersolvable.
\end{theorem}
\begin{proof}
Assume that the theorem is false and let $G$ be a counterexample with $|G|$ minimal. Denote $\mathcal{F}_S (G)$ by $\mathcal{F}$.

Let $H$ be a proper subgroup of $G$ such that $S \cap H \in {\rm Syl}_p (H)$ and $|S \cap H| \geq p^2$. We predicate that $\mathcal{F}_{S \cap H} (H)$ is supersolvable. By our hypothesis, every subgroup of $S \cap H$ with order $p$ is an $NE^{*}$-subgroup of $G$. By lemma {{\cite[Lemma 2.2 (1)]{LZ}}}, every subgroup of $S \cap H$ with order $p$ is an $NE^{*}$-subgroup of $H$. Since $1<p<|S \cap H|$, $H$ satisfies the hypothesis and it follows from the minimal choice of $G$ that $\mathcal{F}_{S \cap H} (H)$ is supersolvable. 
 Now let $Q \in \mathcal{E}_{\mathcal{F}} ^{*}$, Using the same method in Step 2,3 of Theorem \ref{3.1}, we conclude that $|Q| \geq p^2$, $N_{\mathcal{F}} (Q)$ is supersolvable under the assumption that $Q \not\unlhd G$, and $|O_p (G)| \geq p^2$.

Now we prove that $O_p (G) \leq Z_{\mathfrak{U}} (G)$. It follows from $|O_p (G)| \geq p^2$ that there exists a subgroup $D_0 \leq O_p (G)$ with order $p$ such that every subgroup of $O_p (G)$ with order $p$ is an $NE^{*}$-subgroup of $G$. By Lemma \ref{31}, it yields that $O_p (G) \leq Z_{\mathfrak{U}} (G)$ and this part is complete. Using the same method in Step 5 of Theorem \ref{3.1}, we can derive a contradiction, and the proof is complete.
\end{proof}

Now we would like to give some characterizations for supersolvability of $\mathcal{F}_S (G)$, which can be seen as corollaries from our main theorems in section \ref{Section 4}. We will prove these corollaries in the next section with the help of stronger results.
\begin{corollary}\label{Corollary 1}
Let $G$ be a finite group and $S$ be a Sylow $p$-subgroup of $G$. Assume that there exists a subgroup $D$ of $S$ such that $1 <|D|<|S|$, and any subgroup of $S$ of order $|D|$ is abelian and satisfies the $S$-subnormalizer condition in $G$, any subgroup of $S$ of order $p|D|$ is abelian. If $S$ is non-abelian, $p=2$ and $|D|=2$, then assume moreover that any cyclic subgroup of $S$ of order $4$ satisfies the $S$-subnormalizer condition in $G$. Then $\mathcal{F}_S (G)$ is supersolvable.
\end{corollary}
\begin{corollary}\label{Corollary 2}
Let $G$ be a finite group and $S$ be a Sylow $p$-subgroup of $G$, where $p$ is an odd prime divisor of $|G|$. Suppose that any maximal subgroup of $S$ satisfies the $S$-subnormalizer condition in $G$, and if $S$ is not cyclic, we have additionally that $S$ has more than one abelian maximal subgroup. Then $\mathcal{F}_S (G)$ is supersolvable.
\end{corollary}
\begin{corollary}\label{Corollary 3}
Let $G$ be a finite group, $p$ be an odd prime, $S$ be a Sylow $p$-subgroup of $G$. Assume that there is a subgroup $D$ of $S$ such that any subgroup of $S$ of order $1 < |D| <|S|$ is abelian and satisfies the $S$-subnormalizer condition in $G$. Then $\mathcal{F}_S (G)$ is supersolvable. 
\end{corollary}
\section{Characterizations for supersolvability of saturated fusion systems}\label{Section 4}
In this section, we give characterizations for supersolvability of saturated fusion system $\mathcal{F}$ over $S$ (which even holds if $\mathcal{F}$ is exotic) under the assumption that certain subgroups of $S$ are semi-invariant in $\mathcal{F}$ as our main results in this paper. The proofs of the following theorems refer to the proofs in \cite{FJ}. Since all of the conclusions can be parallel migrated,   there is a strong relationship between semi-invariant and weakly $\mathcal{F}$-closed, as there are lots of similarities between the two concepts.
\begin{theorem}\label{Theorem B}
Let $\mathcal{F}$ be a saturated fusion system on a finite $p$-group $S$. Suppose that there is a subgroup $D$ of $S$ with $1<D<S$ such that any subgroup of $S$ of order $|D|$ is abelian and semi-invariant in $\mathcal{F}$, any subgroup of $S$ of order $p|D|$ is abelian. If $S$ is non-abelian, $p=2$ and $|D|=2$, then assume moreover that any cyclic subgroup of $S$ of order $4$ is semi-invariant in $\mathcal{F}$.  Then $\mathcal{F}$ is supersolvable.
\end{theorem}
\begin{proof}
Assume that the theorem is false, and without loss of generality,  let $\mathcal{F}$ be a counterexample such that $|\mathcal{F}|$, which denotes the number of morphisms, is minimal. For the ease of reading, we break the proof into the following steps.
\begin{itemize}
\item[\textbf{Step 1.}] Let $R$ be a subgroup of $S$ with $|R| \geq p|D|$ and $\mathcal{F}_0$ be a proper saturated subsystem of $\mathcal{F}$ on $R$. Then $\mathcal{F}_0$ is supersolvable.
\end{itemize}

It follows immediately from the hypothesis that any subgroup of $R$ of order $|D|$ is abelian and semi-invariant in $\mathcal{F}$, any subgroup of $R$ of order $p|D|$ is abelian if $R$ is non-abelian. Now assume that $H$ is a subgroup of $S$ such that $H$ is semi-invariant in $\mathcal{F}$. Then for any $H \leq K \leq S$ and any $\alpha \in {\rm Aut}_{\mathcal{F}} (K)$, $H^{\alpha} = H$. Therefore  for any $H \leq K \leq R \leq S$ and $\alpha_0 \in {\rm Aut}_{\mathcal{F}_0} (K) \subseteq {\rm Aut}_{\mathcal{F}} (K)$, $H^{\alpha_0} = H$. Hence we conclude that $H$ is semi-invariant in $\mathcal{F}_0$. Consequently, any subgroup of $R$ of order $|D|$ is abelian and semi-invariant in $\mathcal{F}_0$, any subgroup of $R$ of order $p|D|$ is abelian, which indicates that $\mathcal{F}_0$ satisfies the hypothesis of the theorem and so $\mathcal{F}_0$ is supersolvable by the minimal choice of $\mathcal{F}$.
\begin{itemize}
\item[\textbf{Step 2.}]  Let $Q \in \mathcal{E}_0 ^{*}$. Then $|Q| \geq p|D|$, and suppose moreover that $Q$ is not normal in $\mathcal{F}$, then $N_{\mathcal{F}} (Q)$ is supersolvable.  
\end{itemize}

Assume that $|Q| <p |D|$. Then there is a subgroup $R$ of $S$ such that $|R| = p|D|$ and $Q < R$. By our hypothesis, $R$ is abelian, and so $R \leq C_S (Q)$. Since $Q$ is a member of $\mathcal{E}_0 ^{*}$, we have that $Q$ is $\mathcal{F}$-centric. Hence it follows that $R \leq C_S (Q) = Z(Q) \leq Q$, a contradiction. Thus $|Q| <p|D|$.

Now suppose that $Q$ is not normal in $\mathcal{F}$. Then clearly $N_{\mathcal{F}} (Q)$ is a proper saturated fusion subsystem of $\mathcal{F}$ on $N_{\mathcal{F}} (Q)$. Note that $|N_S (Q)| \geq |Q| \geq p|D|$, it follows directly from Step 1 that $N_{\mathcal{F}} (Q)$ is supersolvable. Now we claim that  $|O_p (\mathcal{F})| \geq p|D|$.  Suppose that there is no $Q \in \mathcal{E}_0 ^{*}$ such that $Q \unlhd \mathcal{F}$. Therefore $N_{\mathcal{F}} (Q)$ is supersolvable for every $Q \in  \mathcal{E}_0 ^{*}$. By Lemma \ref{1}, we conclude that $\mathcal{F}$ is supersolvable, a contradiction. Hence there exists a subgroup $Q \in  \mathcal{E}_0 ^{*}$ such that $Q$ is normal in $\mathcal{F}$. Thus we get that $|O_p (\mathcal{F})| \geq |Q| \geq p|D|$ by Step 2.
\begin{itemize}
\item[\textbf{Step 3.}] There exists a series $1 = S_0 \leq S_1 \leq \cdots \leq S_n = O_p (\mathcal{F})$ such that $S_i$ is strongly $\mathcal{F}$-closed, $i=0,1,\cdots,n$, and $S_{i+1} / S_i$ is cyclic, $i=0,1,\cdots,n-1$.
\end{itemize}

Now let $N:=O_p (\mathcal{F})$, $A:={\rm Aut}_{\mathcal{F}} (N)$, and $G = N \rtimes A$ be the outer semi-direct product of $N$ and $A$, where the action induced by $A$ on $N$ is natural. Without loss of generality, we may consider $N$ and $A$ as subgroups of $G$, i.e. $G = NA$, $N \cap A =1$. Clearly, for $x \in N$ and $\alpha \in A$, $x^{\alpha} \in N$.

Now let $P$ be a subgroup of $N$ such that $P$ is semi-invariant in $\mathcal{F}$. We claim that $P \unlhd G$. It is clear that $P \unlhd S$, therefore we have $S \leq N_G (P)$. Since $P \leq N$, we assert from the definition of $\mathcal{F}_0$-semi-invariance that $P^{\alpha} =P$ for any $\alpha \in {\rm Aut}_{\mathcal{F}} (N)=A$. Hence $A \leq N_G (P)$, and so $P$ is normal in $G$.

Since any subgroup of $N$ of order $|D|$ is semi-invariant in $\mathcal{F}$, we conclude from the proceeding paragraph that any subgroup of $N$ of order $|D|$ is normal in $G$. Suppose that $N$ is a non-abelian $2$-group and $|D|=2$, then any cyclic subgroup of $N$ of order $4$ is semi-invariant in $\mathcal{F}$, which yields that any cyclic subgroup of $N$ of order $4$ is normal in $G$. Hence we assert from Lemma \ref{5} that $N \leq Z_{\mathfrak{U}} (G)$. Thus there exists a series $1 = S_0 \leq S_1 \leq \cdots \leq S_n =N$ such that $S_i \unlhd G$, $i=0,1,\cdots,n$, $S_{i+1}/S_i$ is cyclic, $i=0,1,\cdots,n-1$. Now we claim that $S_i$ is strongly $\mathcal{F}$-closed, $i=0,1,\cdots,n$. Let $P = \langle x \rangle$ be a cyclic subgroup of $S_i$, $P_1 = \langle y \rangle$ be a cyclic subgroup of $S$ and $\phi \in {\rm Iso}_{\mathcal{F}} (P,P_1)$, where $x^{\phi} = y$. Since $P \leq S_i \leq N$ and $N$ is strongly $\mathcal{F}$-closed, it follows that $P_1 \leq N$. Note that $N$ is normal in $\mathcal{F}$, there exists an extension $\overline{\phi} \in {\rm Hom}_{\mathcal{F}} (N,N) = A$ of $\phi$. Hence we conclude that $P_1 = P^{\phi} = P^{\overline{\phi}} \leq (S_i)^{\overline{\phi}}$. It indicates from $S_i \unlhd G = NA$ that $(S_i)^{\overline{\phi}} = S_i$, and so $P_1 \leq S_i$, which implies that $S_i$ is strongly $\mathcal{F}$-closed and this step is complete.
\begin{itemize}
\item[\textbf{Step 4.}] Final contradiction. 
\end{itemize}
It follows from Step 1, Step 2, Step 3 that $\mathcal{F}$ satisfies the hypothesis of Lemma \ref{6}. Therefore $\mathcal{F}$ is supersolvable, a contradiction and the proof is complete.
\end{proof}

\begin{theorem}\label{C}
Let $p$ be an odd prime, and $\mathcal{F}$ be a saturated fusion system over a $p$-group $S$. Suppose that any maximal subgroup of $S$ is semi-invariant in $\mathcal{F}$, and if $S$ is not cyclic, we have additionally that $S$ has more than one abelian maximal subgroup. Then $\mathcal{F}$ is supersolvable.
\end{theorem}
\begin{proof}
Assume that our theorem is false, and without loss of generality, let $\mathcal{F}$ be a counterexample such that $|\mathcal{F}|$ is minimal. For the ease of reading, we divide our argument into the following steps.
\begin{itemize}
\item[\textbf{Step 1.}] $S$ is not normal in $\mathcal{F}$.
\end{itemize}

Now suppose that $S \unlhd \mathcal{F}$. Then clearly $\mathcal{F}$ is constrained. Hence we conclude from Model theorem for constrained fusion systems in {{\cite[Part III, Theorem 5.10]{AK}}} that there exists a $p$-closed finite group $G$ with $S \in {\rm Syl}_p (G)$ and $\mathcal{F} = \mathcal{F}_S (G)$. Since any maximal subgroup $H$ of $S$ is semi-invariant in $\mathcal{F}$, $H$ is invariant under the action induced by ${\rm Aut}_{\mathcal{F}} (S) = {\rm Aut}_{\mathcal{F}_S (G)} (S) $. As $S$ is normal in $G$,  it yields that $H$ is normal in $G$ as well. By Lemma \ref{5}, it follows from $p$ is odd that $S \leq Z_{\mathfrak{U}} (G)$. Thus there exists a series of subgroups of $S$, namely
$$1 =S_0 \leq S_1 \leq \cdots \leq S_{m-1} \leq S_m =S,$$
such that $S_i \unlhd G$, $i=0,1,\cdots,m$, $S_{i+1} / S_i$ is cyclic, $i=0,1,\cdots, m-1$. Since $S_i$ is normal in $G$, it implies that $S_i$ is strong $\mathcal{F}$-closed, $i=0,1,\cdots,m$. By the definition of supersolvability, we assert that $\mathcal{F}$ is supersolvable, a contradiction and so $S$ is not normal in $\mathcal{F}$.
\begin{itemize}
\item[\textbf{Step 2.}] $S$ possesses at least two abelian normal subgroups, and $Z(S)$ is contained in any abelian maximal subgroup of $S$. Moreover, $|[S,S]|=p$, $p^2 |Z(S)| = |S|$.
\end{itemize}

Suppose that $S$ is abelian. It indicates from $S$ is strongly $\mathcal{F}$-closed and {{\cite[Part I, Corollary 4.7(a)]{AK}}} that $S$ is normal in $\mathcal{F}$, a contradiction to Step 1. Therefore $S$ is not abelian, and so $S$ is not cyclic. By our hypothesis, $S$ possesses more than one abelian maximal subgroup. Now let $R$ be an abelian maximal subgroup of $S$. If $Z(S)$ is not contained in $R$, then $Z(S) R = S$, and so $S$ is abelian, a contradiction. Thus $Z(S) \leq R$ for any abelian maximal subgroup $R$ of $S$. By {{\cite[Lemma 1.9]{BO}}}, we conclude directly that  $|[S,S]|=p$ and $p^2 |Z(S)| = |S|$.
\begin{itemize}
\item[\textbf{Step 3.}] $O_p (\mathcal{F})$ is an abelian maximal subgroup of $S$, $\mathcal{E}_{\mathcal{F}} ^{*} = \lbrace O_p (\mathcal{F}),S \rbrace$, and $O_p (\mathcal{F})$, $S$ are exactly the subgroups of $S$ which are both $\mathcal{F}$-centric and $\mathcal{F}$-radical.
\end{itemize}

Let $Q$ be a member of $\mathcal{E}_{\mathcal{F}} ^{*}\setminus \lbrace  S \rbrace$. Since $Q$ is $\mathcal{F}$-centric, we conclude that $Z(S) \leq C_S (Q) = Z(Q) \leq Q$. If $Z(S) = Q$, then $C_S (Q) = S = Z(Z(S)) = Z(S)$, a contradiction. Hence $Z(S) <Q<S$, and it follows from $p^2 |Z(S)| = |S|$ in Step 2 that $Q$ is a maximal subgroup of $S$. Since $Q / Z(S) \cong C_p$, it yields that $Q$ is abelian.

Now let $Q \in \mathcal{E}_{\mathcal{F}} ^{*}$ such that $Q$ is not normal in $\mathcal{F}$. Then $N_{\mathcal{F}} (Q)$ is a proper saturated fusion subsystem of $\mathcal{F}$ on $S$ since $Q$ is a maximal subgroup of $S$. Note that for any $T \leq S$, we have ${\rm Aut}_{N_{\mathcal{F}} (Q)} (T) \subseteq {\rm Aut}_{\mathcal{F}} (T)$. Hence for any $H \leq S$ such that $H$ is semi-invariant in $\mathcal{F}$, i.e. $H$ is invariant under the action induced by ${\rm Aut}_{\mathcal{F}} (K)$ for any $H \leq K \leq S$, we get that $H$ is invariant under the action induced by ${\rm Aut}_{N_{\mathcal{F}}(Q)} (K)$ for any $H \leq K \leq S$, which indicates that $H$ is semi-invariant in $N_{\mathcal{F}}(Q)$. Thus $N_{\mathcal{F}}(Q)$ satisfies the hypothesis of the theorem, and by the minimal choice of $\mathcal{F}$, it yields that $N_{\mathcal{F}}(Q)$ is supersolvable. If all members of $\mathcal{E}_{\mathcal{F}} ^{*}$ is not normal in $\mathcal{F}$, it implies that $N_{\mathcal{F}}(Q)$ is supersolvable for all $Q \in \mathcal{E}_{\mathcal{F}} ^{*}$. By Lemma \ref{1}, $\mathcal{F}$ is supersolvable, a contradiction. Hence there exists $Q \in \mathcal{E}_{\mathcal{F}} ^{*}$ such that $Q$ is normal in $\mathcal{F}$. By Step 1, $Q <S$. Hence $Q$ is an abelian maximal subgroup of $S$. It is clear that $Q = O_p (\mathcal{F})$. Thus $O_p (\mathcal{F})$ is an abelian maximal subgroup of $S$.

Clearly $S$ is $\mathcal{F}$-centric and $\mathcal{F}$-radical. Since $O_p (\mathcal{F})$ is essential, it follows from {{\cite[Part I, Proposition 3.3(a)]{AK}}} that $O_p (\mathcal{F})$ is $\mathcal{F}$-centric and $\mathcal{F}$-radical. On the other hand, Let $R$ be a subgroup of $S$ such that $R$ is $\mathcal{F}$-centric and $\mathcal{F}$-radical. Then we conclude from {{\cite[Part I, Proposition 4.5]{AK}}} that $O_p (\mathcal{F}) \leq R$. The maximality of $O_p (\mathcal{F})$ implies that $R$ is equal to $S$ or $O_p (\mathcal{F})$. Hence the subgroups of $S$ which are $\mathcal{F}$-centric and $\mathcal{F}$-radical are exactly $S$ and $O_p (\mathcal{F})$.

Finally, it follows again from {{\cite[Part I, Proposition 3.3(a)]{AK}}} that $\mathcal{E}_{\mathcal{F}} ^{*} \subseteq \lbrace O_p (\mathcal{F}),S\rbrace$, which yields that  $\mathcal{E}_{\mathcal{F}} ^{*} = \lbrace O_p (\mathcal{F}),S\rbrace$.
\begin{itemize}
\item[\textbf{Step 4.}] There is no non-trivial subgroup of $Z(S)$ which is normal in $\mathcal{F}$.  
\end{itemize}

Assume that there exists a subgroup $Z$ of $Z(S)$ such that $1 <Z \leq Z(S)$ and $Z \unlhd \mathcal{F}$. Let $H/Z$ be any maximal subgroup of $S/Z$. Clearly $H$ is a maximal subgroup of $S$, and we conclude from the hypothesis of the theorem that $H$ is semi-invariant in $\mathcal{F}$. By Lemma \ref{7}, $H/Z$ is semi-invariant in $\mathcal{F}/Z$. By Step 2 and $Z(S) < Q$ for every $Q \in \mathcal{E}_{\mathcal{F}} ^{*}$, we get that there are at least two abelian maximal subgroups in $S/Z$. By {{\cite[Part I, Proposition 5.11]{CR}}}, we have that $\mathcal{F}/Z$ is saturated.  Therefore we obtain that $\mathcal{F}/Z$ satisfies the hypothesis of the theorem and so $\mathcal{F}/Z$ is supersolvable by the choice of $\mathcal{F}$. 

Now let $R_1,R_2$ be two distinct abelian normal subgroup of $S$. Since $R_1/Z$, $R_2/Z$ are semi-invariant in $\mathcal{F}$, $\mathcal{F}/Z$ is supersolvable, we obtain from Lemma \ref{9} that $R_1 /Z$, $R_2 /Z$ are strongly $\mathcal{F}/Z$-closed. By Lemma \ref{10}, we assert that $R_1$ and $R_2$ are strongly $\mathcal{F}$-closed. Note that $R_1,R_2$ are abelian, it follows from {{\cite[Part I, Corollary 4.7(a)]{AK}}} that $R_1$, $R_2$ are normal in $\mathcal{F}$. Hence $R_1 R_2 = S$ is normal in $\mathcal{F}$, a contradiction. Thus we conclude that there is no non-trivial subgroup of $Z(S)$ which is normal in $\mathcal{F}$ and this step is complete.
\begin{itemize}
\item[\textbf{Step 5.}] The Sylow $p$-subgroups of ${\rm Aut}_{\mathcal{F}} (O_p (\mathcal{F}))$ have order $p$, and ${\rm Aut}_{\mathcal{F}} (O_p (\mathcal{F}))$ is not $p$-closed.
\end{itemize} 

Since $O_p (\mathcal{F})$ is normal in $\mathcal{F}$, $O_p (\mathcal{F})$ is weakly $\mathcal{F}$-closed, i.e. $O_p (\mathcal{F})^{\mathcal{F}} = \lbrace O_p (\mathcal{F}) \rbrace$. 
As $\mathcal{F}$ is saturated, we conclude from the definition of saturated fusion system that $O_p (\mathcal{F})$ is fully normalized in $\mathcal{F}$, i.e. ${\rm Aut}_S (O_p (\mathcal{F})) \in {\rm Syl}_p ({\rm Aut}_{\mathcal{F}} (O_p (\mathcal{F})))$. By Step 3, we get that $O_p (\mathcal{F})$ is $\mathcal{F}$-centric, abelian and maximal in $S$. Hence we obtain that ${\rm Aut}_S (O_p (\mathcal{F})) \cong S/O_p (\mathcal{F}) \cong C_p$. Again, it indicates from Step 3 that $O_p (\mathcal{F})$ is $\mathcal{F}$-radical. Hence we assert from $O_p (\mathcal{F})$ is abelian that $1 = O_p ({\rm Out}_{\mathcal{F}} (O_p (\mathcal{F})))=O_p ({\rm Aut}_{\mathcal{F}} (O_p (\mathcal{F})))={\rm Inn} (O_p (\mathcal{F}))$. Thus ${\rm Aut}_{\mathcal{F}} (O_p (\mathcal{F}))$ is not $p$-closed. Now Suppose that $\alpha \in {\rm Aut}_{\mathcal{F}} (O_p (\mathcal{F}))$ has order $p$. Then we claim that the commutator group $[\alpha,O_p (\mathcal{F})]$ has order $p$. Actually, since ${\rm Aut}_S (O_p (\mathcal{F}))$ is a Sylow $p$-subgroup of  ${\rm Aut}_{\mathcal{F}} (O_p (\mathcal{F}))$, it follows from Sylow Theorem that there exists $\beta \in {\rm Aut}_{\mathcal{F}} (O_p (\mathcal{F}))$ such that $\alpha ^{\beta} := \gamma \in {\rm Aut}_{S} (O_p (\mathcal{F}))$. Note that $[\alpha,O_p (\mathcal{F})]^{\beta} = [\alpha^{\beta},O_p (\mathcal{F})]$, hence without loss of generality we may assume that $\alpha \in {\rm Aut}_{S} (O_p (\mathcal{F}))$. Then there exists $s \in S$ such that $x^{\alpha} = x^s$ for any $x \in O_p (\mathcal{F})$. Therefore we assert that $[\alpha,O_p (\mathcal{F})]=[s,O_p (\mathcal{F})] \leq S'$. As $\alpha$ has order $p$, $\alpha$ acts non-trivially on $O_p (\mathcal{F})$. Hence $[\alpha,O_p (\mathcal{F})] \neq 1$. By Step 2, we conclude that $|[\alpha,O_p (\mathcal{F})]|=p$, and the proof of this step is complete.
\begin{itemize}
\item[\textbf{Step 6.}] $S$ is extraspecial of order $p^3$ and exponent $p$.
\end{itemize}

We give some notations first. Let $A:=O_p (\mathcal{F})$, $G:={\rm Aut}_{\mathcal{F}} (O_p (\mathcal{F}))$, $A_1 := C_A (H)$, $H:=O^{p'} (G)$ and $A_2 :=[H,A]$. Combining the fact that $G \leq  {\rm Aut} (O_p (\mathcal{F}))$, Step 5  and $A$ is abelian, we obtain from Lemma \ref{8} that $G$ normalizes $A_1$ and $A_2$, $A = A_1 \times A_2$, $A_2 = C_p \times C_p$. 

Since ${\rm Aut}_S (O_p (\mathcal{F}))$ is a Sylow $p$-subgroup of $G$, it follows from $H = O^{p'} (G)$ that ${\rm Aut}_S (A) \leq H$. Hence we get that ${\rm Aut}_S (A)$ centralizes $A_1$. As $A = O_p (\mathcal{F})$ is normal in $S$, for any $s \in S$ and $a \in A_1$, we have $a^s = a$. Therefore we conclude that $A_1$ lies in the center of $S$. 

Now we claim that $A_1 \unlhd \mathcal{F}$. Let $Q$ be a subgroup of $S$ such that $Q$ is $\mathcal{F}$-essential. By {{\cite[Part I, Corollary 3.3(a)]{AK}}}, $Q$ is $\mathcal{F}$-centric and $\mathcal{F}$-radical. By step 4, $Q = A$ or $Q = S$. Since $G$ normalizes $A_1$, $A_1$ is invariant in ${\rm Aut}_{\mathcal{F}} (O_p (\mathcal{F}))=G$. It follows from {{\cite[Part I, Corollary 3.3(a)]{AK}}} that we only need to prove that $A_1$ is ${\rm Aut}_{\mathcal{F}} (S)$-invariant. However, this follows directly from the fact that $G$ normalizes $A_1$ because any $\mathcal{F}$-automorphism of $S$ restricts to an $\mathcal{F}$-automorphism of $A$.

Note that $A_1$ is now normal in $\mathcal{F}$ by the argument in the proceeding paragraph. By Step 4, we get that $A_1 = 1$. Hence $A = A_2 \cong C_p \times C_p$. It indicates from the maximality of $A$ that $|S| = p^3$. As $S$ is non-abelian which is showed in Step 2 and any non-abelian $p$-group of order $p^3$ is extraspecial, we conclude that $S$ is extraspecial. 

Now we prove that $S$ has exponent $p$. Suppose that the exponent of $S$ is not equal to $p$, it follows from $A \cong C_p \times C_p$ that $S$ has exponent $p^2$. Hence $S$ has a cyclic maximal subgroup and so $S$ is metacyclic. However, it follows from {{\cite[Theorem C]{BO}}} and $p$ is odd that $\mathcal{F}$ is supersolvable, a contradiction and so $S$ has exponent $p$.
\begin{itemize}
\item[\textbf{Step 7.}] Final contradiction.
\end{itemize}

It yields from Step 6 that $S$ is extraspecial of order $p^3$ and exponent $p$. As we have argued in Step 3, $S$ has exactly one subgroup $A$ which is elementary abelian, $\mathcal{F}$-centric, and $\mathcal{F}$-radical. Then we conclude from {{\cite[Lemma 4.7]{RV}}} that $\mathcal{F}$ is isomorphic to one of the fusion systems considered in {{\cite[Lemma 2.11]{FJ}}}. Therefore we conclude again from {{\cite[Lemma 2.11]{FJ}}} that there exists a subgroup $P$ of $S$ which is not weakly $\mathcal{F}$-closed. Then clearly $P \neq A$, and there exists $\phi \in {\rm Iso}_{\mathcal{F}} (P,R)$ such that $R \neq P$. Obviously, $R$ is not equal to $A$ as well. Since $A$ is normal in $\mathcal{F}$, $\phi$ can be extended to $\overline{\phi} \in {\rm Hom}_{\mathcal{F}}
(PA,RA) = {\rm Aut}_{\mathcal{F}} (S)$. As $P$ is a maximal subgroup of $S$, we conclude that $P$ is invariant in ${\rm Aut}_{\mathcal{F}} (S)$, i.e. $P = P^{\overline{\phi}} = P^{\phi} = R$, a contradiction. This final contradiction finishes the proof and we are done.
\end{proof}
\begin{theorem}\label{D1}
Let $G$ be a finite group, $p$ be an odd prime, $S$ be a Sylow $p$-subgroup of $G$ and $\mathcal{F}:=\mathcal{F}_S (G)$. Assume that there is a subgroup $D$ of $S$ such that any subgroup of $S$ of order $1 < |D| <|S|$ is abelian and semi-invariant in $\mathcal{F}$. Then $\mathcal{F}$ is supersolvable. 
\end{theorem}
\begin{proof}
Suppose that the theorem is false, and without loss of generality,  let $G$ be a counterexample of minimal order. For the ease of reading, we break the argument into the following steps.
\begin{itemize}
\item[\textbf{Step 1.}]  Let $H <G$ with $S \cap H\in {\rm Syl}_p (H)$ and $|S \cap H| >|D|$. Then $\mathcal{F}_{S \cap H} (H)$ is supersolvable.
\end{itemize}

It follows immediately from the hypothesis that any subgroup of $S \cap H$ of order $|D|$ is abelian and semi-invariant in $\mathcal{F}_S (G)$. Now assume that $T$ is a subgroup of $S$ such that $T$ is semi-invariant in $\mathcal{F}_S (G)$. Then for any $T \leq K \leq S$ and any $\alpha \in {\rm Aut}_{\mathcal{F}} (K)$, $T^{\alpha} = T$. Therefore for any $T \leq K \leq S \cap H$ and $\alpha_0 \in {\rm Aut}_{\mathcal{F}_{S \cap H} (H)} (K) \subseteq {\rm Aut}_{\mathcal{F}} (K)$, $T^{\alpha_0} = T$. Hence we conclude that $T$ is semi-invariant in $\mathcal{F}_{S \cap H} (H)$. Consequently, any subgroup of $S \cap H$ of order $|D|$ is abelian and semi-invariant in $\mathcal{F}_{S \cap H} (H)$, which indicates that $\mathcal{F}_{S \cap H} (H)$ satisfies the hypothesis of the theorem and so $\mathcal{F}_{S \cap H} (H)$ is supersolvable by the minimal choice of $G$.
\begin{itemize}
\item[\textbf{Step 2.}] $C_G (O_p (G)) \leq O_p (G)$. 
\end{itemize}

Let $Q \in \mathcal{E}_{\mathcal{F}} ^{*}$. As $Q$ is $\mathcal{F}$-centric, we have $C_S (Q)\leq Q$. Since any subgroup of $S$ of order $|D|$ is abelian, we obtain that $|Q| \geq |D|$. Hence $|N_S (Q)|>|D|$. 

Suppose that $Q$ is not normal in $G$. Then clearly $N_G (Q) <G$. If $Q =S$, then $N_S (Q) = S \in {\rm Syl}_p (N_G (Q))$. If $Q <S$, then $Q$ is $\mathcal{F}$-essential. By {{\cite[Proposition 3.3(a)]{AK}}}, $Q$ is fully normalized in $\mathcal{F}$. It follows from {{\cite[Definition 2.4]{AK}}} that $N_S (Q) \in {\rm Syl}_p (N_G (Q))$. From the proceeding paragraph, we get that $|S \cap N_G (Q)| = |N_S (Q)| >|D|$. By Step 1, we conclude that $\mathcal{F}_{N_S(Q)} (N_G (Q))= N_{\mathcal{F}} (Q)$ is supersolvable.

Assume that no member of $\mathcal{E}_{\mathcal{F}} ^{*}$ is normal in $G$, then $N_{\mathcal{F}} (Q)$ is supersolvable for any $Q \in \mathcal{E}_{\mathcal{F}} ^{*}$. Hence we assert from Lemma  \ref{1} that $\mathcal{F}$ is supersolvable, a contradiction. Therefore there exists $Q \in \mathcal{E}_{\mathcal{F}} ^{*}$ such that $Q \unlhd G$. If $Q =S$, then $Z(Q) =Z(S) = C_S (Q)$, i.e. $Q$ is $\mathcal{F}$-centric. If $Q<S$, then $Q$ is $\mathcal{F}$-centric. Thus $\mathcal{F}$ is constrained. By {{\cite[Proposition 8.8]{LM}}}, the unique model of $\mathcal{F}_S (G) = \mathcal{F}$, namely $L$, is isomorphic to $\overline{G}:=G/O_{p'} (G)$.

Now suppose that $O_{p'} (G) >1$, and let $\phi: L \rightarrow G/O_{p'} (G)$ be an automorphism from $L$ to $G/O_{p'} (G)$. Clearly the image of $S$ can be written as $S_1 O_{p'} (G)/O_{p'} (G)$, where $S_1 \leq G$ is a Sylow $p$-subgroup of $G$. Now we claim that $\mathcal{F}_{S_1 O_{p'} (G)/O_{p'} (G)} (G/O_{p'} (G))$ is supersolvable. Since for any $P,Q \leq S$, we have the following:
\begin{align}\label{10001}
{\rm Hom}_{\mathcal{F}_S (G)} (P,Q) ={\rm Hom}_{\mathcal{F}_S (L)} (P,Q) = \left\{ \phi \psi \phi^{-1},P \rightarrow Q \,\Big|\,\psi\in  {\rm Hom}_{\mathcal{F}_{\overline{S}_1} (\overline{G})} (P^{\phi},Q^{\phi})\right\}.
 \end{align}
In particular, we have the following equation for any $P \leq S$:
\begin{align*}
{\rm Aut}_{\mathcal{F}_S (G)} (P) = {\rm Aut}_{\mathcal{F}_S (L)} (P) = \left\{ \phi \psi \phi^{-1},P \rightarrow P \,\Big|\,\psi\in  {\rm Aut}_{\mathcal{F}_{\overline{S}_1} (\overline{G})} (P^{\phi})\right\}.
\end{align*}
Now let $H$ be an arbitrary subgroup of $\overline{S}_1$ such that $|H| =|D|$. Then for any $\alpha \in {\rm Aut}_{\mathcal{F}_{\overline{S}_1} (\overline{G})} (K^{\phi})$ such that $H \leq K^{\phi} \leq \overline{S}_1$, it follows that $H^{\phi ^{-1}} \leq K \leq S$, $|H^{\phi ^{-1}}| = |H|=|D|$, $\phi \alpha \phi^{-1} \in {\rm Aut}_{\mathcal{F}_S (G)} (K)$ and $H$ is semi-invariant in $\mathcal{F}$ that $(H^{\phi ^{-1}})^{\phi \alpha \phi^{-1}} = H^{\phi ^{-1}}$, i.e. $H^{\alpha} = H$. By the choice of $K$ and $\alpha$, we conclude that $H$ is semi-invariant in $\mathcal{F}_{\overline{S}_1} (\overline{G})$. Since $H^{\phi ^{-1}} \leq S$ is abelian, $H$ is abelian as well. By the arbitrariness of $H$, we assert that $\mathcal{F}_{\overline{S}_1} (\overline{G})$ satisfies the hypothesis of the theorem. By the minimal choice of $G$, $\mathcal{F}_{\overline{S}_1} (\overline{G})$ is supersolvable. By (\ref{10001}), it yields that both $\mathcal{F}_S (L)$ and $\mathcal{F}_S (G)$ are supersolvable, a contradiction. Hence $O_{p'} (G) =1$, and so  the unique model of $\mathcal{F}_S (G)$ is isomorphic to $G$, i.e. $G$ is the model of $\mathcal{F}$. Thus we have $C_G (O_p (G)) \leq O_p (G)$ and this step is complete.
\begin{itemize}
\item[\textbf{Step 3.}] $|O_p (G)| = |D|$.
\end{itemize}

Suppose that $|O_p (G)| >|D|$. Then we get from the hypothesis of the theorem that any subgroup $H$ of $O_p (G)$ of order $|D|$ is semi-invariant in $\mathcal{F}$. Note that $H$ is invariant in ${\rm Aut}_{\mathcal{F}_S (G)} (O_p (G))$, we conclude immediately that $H$ is normal in $G$. Hence we obtain from any subgroup of $O_p (G)$ of order $|D|$ is normal in $G$ and Lemma \ref{5} that $O_p (G) \leq Z_{\mathfrak{U}} (G)$. Let $L$ be a proper subgroup of $G$ such that $O_p (G) < S \cap L$ and $S \cap L \in {\rm Syl}_p (L)$. Then we have $|S \cap H| > |O_p (G)| > |D|$. By Step 1, we get that $\mathcal{F}_{S \cap L} (L)$ is supersolvable. By the arbitrariness of $L$, it yields from Lemma \ref{2} that $\mathcal{F}_S (G)$ is supersolvable, which is absurd and we get that $|O_p (G)| \leq |D|$. Since $C_G (O_p (G)) \leq O_p (G)$, it follows that $|D| \leq |O_p (G)|$, which indicates that $|D| = |O_p (G)|$.
\begin{itemize}
\item[\textbf{Step 4.}] $O_p (G)$ is an elementary abelian  subgroup of $G$.
\end{itemize}

Assume that $O_p (G)$ is not elementary abelian, then we get from {{\cite[Lemma 4.5]{IS}}} that $\Phi (O_p (G)) \neq 1$. Now let $\overline{G}:= G/\Phi (O_p (G))$ and $\overline{\mathcal{F}}:=\mathcal{F}/\Phi (O_p (G)) = \mathcal{F}_{\overline{S}} (\overline{G})$. Let $H/\Phi (O_p (G))$ be a subgroup of $\overline{S}$ such that $|H/\Phi (O_p (G))| = |D|/|\Phi (O_p (G))|>1$. Then $|H| = |D|$, and so $H$ is semi-invariant in $\mathcal{F}$ by the hypothesis of the theorem. It follows from Lemma \ref{7} that $H/\Phi (O_p (G))$ is semi-invariant in $\overline{\mathcal{F}}$. Clearly $H/\Phi (O_p (G))$ is abelian. By the arbitrariness of $H/\Phi (O_p (G))$, we conclude that any subgroup of $\overline{S}$ of order $|H|/|\Phi (O_p (G))|$ is abelian and semi-invariant in $\overline{\mathcal{F}}$. It follows from the minimality of $G$ that $\mathcal{F}_{\overline{S}} (\overline{G})$ is supersolvable.

We obtain from Lemma \ref{12} that there exists a series
$$\Phi (O_p (G)) \leq V_0 \leq V_1 \leq \cdots \leq V_m =S,$$
such that $\overline{V_i}$ is strongly $\overline{\mathcal{F}}$-closed, $i=0,1,\cdots,m$, $\overline{V_{i+1}}/\overline{V_i}$ is cyclic, $i=0,1,\cdots,m$, and $V_n = O_p (G)$ for some $1 \leq n \leq m-1$ since $O_p (G) / \Phi (O_p (G))$ is normal in $G/\Phi (O_p (G))$, i.e. $O_p (G)/\Phi (O_p (G))$ is normal in $\mathcal{F}_{\overline{S}} (\overline{G})$. Since $\overline{V_i}$ is normal in $\overline{G}$, $i=0,1,\cdots,n$, it follows directly that every chief factor of $\overline{G}$ under $\overline{O_p (G)}$ is of prime order. Hence we obtain from {{\cite[Definition 7.1]{MW}}} that $\overline{O_p (G)}$ is supersolvably embedded in $\overline{G}$. By {{\cite[Theorem 7.19]{MW}}}, it indicates that $O_p (G)$ is supersolvably embedded in $G$. Therefore every chief factor of $G$ under $O_p (G)$ is of prime order, and so $O_p (G) \leq Z_{\mathfrak{U}} (G)$.

It follows from Step 1 and Step 3 that for any proper subgroup $H$ of $G$ with $O_p (G) < S \cap H$, $S \cap H \in {\rm Syl}_p (H)$, the fusion system $\mathcal{F}_{S \cap H} (H)$ is supersolvable. Thus we conclude from Lemma \ref{2} that $\mathcal{F}$ is supersolvable, a contradiction. Hence $O_p (G)$ is elementary abelian and this step is complete.
\begin{itemize}
\item[\textbf{Step 5.}] Suppose that $O_p (G) \leq H <G$. Then $H$ is $p$-closed. 
\end{itemize}

Let $H$ be an arbitrary subgroup of $G$ such that $O_p (G) \leq H <G$. Then we may assume that $S \cap H \in {\rm Syl}_p (H)$. 
It is obvious that $O_p (G) \leq S \cap H$, and if $O_p (G) = S \cap H$, then our proof is complete. Now suppose that $O_p (G) <H \cap S$, then we conclude from Step 1 that $\mathcal{F}_{S \cap H} (H)$ is supersolvable. Hence if yields from {{\cite[Proposition 2.3]{SN}}} that $H \cap S$ is normal in $\mathcal{F}_{S \cap H} (H)$. Therefore we get that $\mathcal{F}_{S \cap H} (H) = N_{\mathcal{F}_{S \cap H} (H)} (S \cap H)$. Now let $h \in H$, and $c_h$ denotes the automorphism of $O_p (G)$ induced by $h$, i.e.
$$c_h: O_p (G) \rightarrow O_p (G), x \mapsto x^h. $$ 
It is clear that $c_h \in \mathcal{F}_{S \cap H} (H) = N_{\mathcal{F}_{S \cap H} (H)} (S \cap H)$, hence there exists $u \in N_H (S \cap H)$ such that $c_u = c_h$, i.e. $hu^{-1} \in C_G (O_p (G)) \leq O_p (G) \leq N_H (S \cap H)$ since Step 2 yields that $G$ is $p$-constrained. Thus we assert that $h \in N_H (S \cap H)$. By the choice of $h$, we conclude that $H = N_H (S \cap H)$, i.e. $H \cap S \unlhd H$. Therefore $H$ is $p$-closed, and   the proof of this step is finished.
\begin{itemize}
\item[\textbf{Step 6.}] $S /O_p (G)$ is cyclic.
\end{itemize}

Let $H$ be a subgroup of $G$ such that $O_p (G) \leq H <G$ and $\overline{G}:=G/O_p (G)$. It yields from Step 5 that $H$ is $p$-closed. Therefore $\overline{H}$ is $p$-closed. Suppose that $\overline{G}$ is $p$-closed, then $G$ is $p$-closed, and so $O_p (G) = S$, which implies that $|D| = |O_p (G) | = |S|$ by Step 3, impossible. Hence $\overline{G}$ is not $p$-closed, and so $\overline{G}$ is a finite group which is not $p$-closed, while any  proper subgroup of $\overline{G}$ is $p$-closed. Therefore $\overline{G}$ is a minimal non-$p$-closed group, and we conclude from {{\cite[Lemma 1]{QX}}} that $\overline{G}$ is minimal non-nilpotent or $\overline{G}/\Phi (\overline{G})$ is non-abelian and simple.

Suppose that the former case holds, then we obtain from Lemma \ref{13} and $S / O_p (G) \in {\rm Syl}_p (\overline{G})$ that $\overline{S}$ is cyclic and this step is complete.

Now suppose that the latter case holds, and so $\widehat{G}:=\overline{G}/\Phi (\overline{G})$ is a non-abelian simple group. Since $\Phi(\overline{G})$ is nilpotent and $O_p (\overline{G})=1$, we get that $\Phi (O_p(\overline{G}))$ is a $p'$-group. Hence the Sylow $p$-subgroups of $\widehat{G}$ are isomorphic to $\overline{S}$, and we only need to show that $\widehat{G}$ has cyclic Sylow $p$-subgroups.

Let $O_p (G) \leq L < G$ such that $L/O_p (G) = \Phi (\overline{G})$ and let $H$ be an arbitrary subgroup of $G$ such that $L \leq H <G$. Then it indicates from Step 5 that $H$ is $p$-closed, hence the quotient group $H/L$ is $p$-closed. It is obvious that $G/L \cong \overline{G} / \overline{L} = \widehat{G}$, therefore we conclude from the Correspondence Theorem that any proper subgroup of $\widehat{G}$ is $p$-closed. Note that the only proper quotient of $\widehat{G}$ is $\widehat{G}/\widehat{G}$ since $\widehat{G}$ is simple, which is obviously $p$-closed. Thus  $\widehat{G}$ is a non-$p$-closed group, where all of its proper subgroups and proper quotient groups are $p$-closed. Hence we obtain that $\widehat{G}$ is a minimal non-$p$-closed group in the sense of {{\cite{LC}}}. Now we assert from {{\cite[Theorem 3.5]{LC}}} that $\widehat{G}$ has cyclic Sylow $p$-subgroups and this step is finished.
\begin{itemize}
\item[\textbf{Step 7.}] $O_p (G)$ is not a maximal subgroup of $S$.
\end{itemize}

Suppose that $O_p (G)$ is a maximal subgroup of $S$, and so we conclude from $|O_p (G)| = |D|$ by Step 3 and the hypothesis of the theorem that every maximal subgroup of $S$ is abelian and semi-invariant in $\mathcal{F}_S (G)$. Then it follows directly from Theorem \ref{C} that $\mathcal{F}_S (G)$ is supersolvable, a contradiction. Now we claim that $|O_p (G)| \geq p^4$.
Actually, it yields from Step 2 and Step 4 that $C_S (O_p (G)) = O_p (G)$. It follows from N-C Theorem that $N_ S (O_p (G))/C_S (O_p (G)) = S/O_p (G)$ is isomorphic to a $p$-subgroup of ${\rm Aut}(O_p (G))$, and $S / O_p (G)$ is cyclic by Step 6. Note that $S \neq O_p (G)$ by Step 3, and $O_p (G)$ is not maximal in $S$ by Step 7, we obtain that $|S/O_p (G)| \geq p^2$.

Now let $|O_p (G)| = p^t$ with $t$ an integer. Since $|D| = |O_p (G)| >1$ by Step 3, we get that $t>0$. It yields from Step 4 that $O_p (G)$ is elementary abelian, hence ${\rm Aut} (O_p (G)) \cong GL_t (p)$. Suppose that $t=1$, then $S/O_p (G)$ is isomorphic to a $p$-subgroup of $C_{p-1}$, a contradiction. Thus $t>1$. Now suppose that $t =2$ or $t=3$. Applying Lemma \ref{14}, we assert that $GL_t (p) \cong {\rm Aut} (O_p (G))$ does not have non-trivial cyclic $p$-subgroup of order greater than $p$. As $S / O_p (G)$ is a cyclic $p$-subgroup of order greater than $p$, we obtain that $t \neq 2,3$. Hence $t \geq 4$ and this part of the proof is complete. 
\begin{itemize}
\item[\textbf{Step 8.}] Final contradiction.
\end{itemize}

Since $|S/O_p (G)| \geq p^2$ as we have shown in the proceeding paragraph, there exists a subgroup $O_p (G) < T \leq S$ such that $|T/O_p (G)| = p^2$. Then $O_p (G)$ is properly contained in a maximal subgroup $T_1$ of $T$. Since $C_{T_1} (O_p (G)) \leq C_G (O_p (G)) \leq O_p (G) <T_1$, it follows that $T_1$ is non-abelian. Note that $|O_p (G) |=|D|$, it indicates from the hypothesis of the theorem that  every subgroup of $T_1$ with index $p^2$ is abelian and semi-invariant in $\mathcal{F}$. Therefore $T$ is actually an $\mathcal{A}_2$-group with $|T| =p^2 |O_p (G)| \geq p^6$ by Step 7.

Suppose that $T$ is metacyclic, then $O_p (G)$ is metacyclic in a natural way. However, we conclude from step 4 that $O_p (G)$ is an elementary abelian subgroup of $T$, which indicates that the order of $O_p (G)$ is at most $p^2$, a contradiction to Step 8. Hence we obtain that $T$ is not metacyclic.

Now suppose that $T$ has no abelian maximal subgroups. Then we assert immediately from Lemma \ref{15} (2) that $|T|=p^5$. Then it yields from $|T/O_p (G)| = p^2$ that $|O_p (G) |=p^3$, a contradiction to Step 8 again. Therefore $T$ has at least one abelian maximal subgroup.

Now assume that there is only one maximal subgroup of $T$ which is abelian. Then denote this abelian maximal subgroup by $U$. If $O_p (G) \leq U$, then $O_p (G)$ is abelian, and it follows from Step 2 that $C_G (O_p (G)) \geq U >O_p (G)$, a contradiction. Hence we get that $O_p (G) \not\leq U$. By the maximality of of $U$, we assert that $T = U O_p (G)$. Now let $Z:= O_p (G) \cap U$. It indicates from the commutativity of $O_p (G)$ and $U$ that $C_T (Z) \geq U O_p (G) =T$, hence $Z \leq Z(T)$. Then we obtain from Isomorphism Theorem that 
$$O_p (G)/Z = O_p (G) / (O_p (G) \cap U )\cong UO_p (G)/U = T/U \cong C_p.$$
Therefore $Z$ is a maximal subgroup of $O_p (G)$. It implies again from Step 2 that $Z(T) \leq C_T (O_p (G)) \leq C_G (O_p (G)) \leq O_p (G)$. Suppose that $Z(T) = O_p (G)$, then $T = UO_p (G) = Z(T) U$ is abelian, a contradiction. Thus $Z(T)<O_p (G)$, and so $Z \leq Z(T) <O_p (G)$, which yields that $Z = Z(T)$. Hence $O_p (G) /Z(T)$ is a normal subgroup of $Z/Z(T)$ with order $p$, and so we obtain that $[O_p (G) /Z(T), Z/Z(T)] < O_p (G) /Z(T)$, which implies that $O_p (G) / Z(T) \leq Z(T/Z(T))$. Since $T/Z(T)/O_p (G)/Z(T) \cong T/O_p (G)$ is abelian, we conclude immediately that $T/Z(T)$ is abelian. Thus we assert that $T' \leq Z(T)$. Now applying Lemma \ref{15} (1), it follows that $Z(T) = \Phi (T) <O_p (G)$. By {{\cite[Lemma 4.5]{IS}}}, we get that $T / O_p (G)$ is elementary abelian. As $|T/O_p (G)| = p^2$, $T/O_p (G)$ is not cyclic, and so $T/O_p (G) \leq S/O_p (G)$ is not abelian, a contradiction to Step 6. Hence $T$ has at least two abelian maximal subgroups.

Applying {{\cite[Lemma 1.9]{BO}}} to $T$, it implies that $|T/Z(T)| = p^2$. Note that $|T/O_p (G)|=p^2$, and $Z(T) \leq C_T (O_p (G)) \leq C_G (O_p (G)) \leq O_p (G)$ by Step 2, we assert that $O_p (G) = Z(T)$. Hence $T \leq C_G (O_p (G)) \leq O_p (G)$, a contradiction and our proof is complete.
\end{proof}
\begin{theorem}
Let $p$ be an odd prime and $\mathcal{F}$ be a saturated fusion system on a finite $p$-group $S$. Assume that there exists a subgroup $D$ of $S$ such that any subgroup of $S$ of order $1<|D| <|S|$ is abelian and semi-invariant in $\mathcal{F}$, then $\mathcal{F}$ is supersolvable.
\end{theorem}
\begin{proof}
Suppose that the theorem is false, and let $\mathcal{F}$ be a counterexample such that $|\mathcal{F}|$ is minimal. Firstly, Let $Q$ be a member of $\mathcal{E}_{\mathcal{F}} ^{*}$ such that $Q$ is not normal in $\mathcal{F}$. Then it is clear that $C_S (Q) \leq Q$, and we conclude from any subgroup of $S$ of order is abelian that $|Q| \geq |D|$. Then $|N_S (Q)| >|D|$. As $Q$ is not normal in $\mathcal{F}$, $N_{\mathcal{F}} (Q)$ is a proper saturated fusion system of $\mathcal{F}$. It is clear that every subgroup of $N_S(Q)$ of order $|D|$ is abelian and semi-invariant in $\mathcal{F}$. Let $H$ be a subgroup of $N_S (Q)$ which is semi-invariant in $\mathcal{F}$. Hence for any $H \leq K \leq S$ and $\alpha \in {\rm Aut}_{\mathcal{F}} (K)$, we have $H^{\alpha} = H$. Therefore for any $H \leq K \leq N_S (Q) \leq S$ and $\alpha \in {\rm Aut}_{N_{\mathcal{F}} (Q)} (K) \subseteq {\rm Aut}_{\mathcal{F}} (K)$, we have $H^{\alpha} = H$. Thus $H$ is semi-invariant in $N_{\mathcal{F}} (Q)$. Now we assert that every subgroup of $N_S(Q)$ of order $|D|$ is abelian and semi-invariant in $N_{\mathcal{F}} (Q)$, and so $N_{\mathcal{F}} (Q)$ satisfies the hypothesis of the theorem. By the minimality of $\mathcal{F}$, $N_{\mathcal{F}} (Q)$ is supersolvable. Suppose that no member of $\mathcal{E}_{\mathcal{F}} ^{*}$ is normal in $\mathcal{F}$, then it follows from Lemma \ref{1} that $\mathcal{F}$ is supersolvable, a contradiction. Hence there exists $Q \in \mathcal{E}_{\mathcal{F}} ^{*}$ such that $Q$ is normal in $\mathcal{F}$. Therefore $\mathcal{F}$ is constrained. Then we obtain from  the  Model Theorem {{\cite[Part III, Theorem 5.10]{AK}}} that there exists a finite group $G$ such that $S \in {\rm Syl}_p (G)$ and $\mathcal{F} = \mathcal{F}_S (G)$. Then we conclude immediately from Theorem \ref{D1} that $\mathcal{F}_S (G) = \mathcal{F}$ is supersolvable, a contradiction. Hence our proof is complete.
\end{proof}

Now with the strong results we have obtained above, we are ready to prove the corollaries which we have listed at the end of  Section \ref{1003}.
\begin{proof}[Proof of Corollary \ref{Corollary 1}]
Since there exists a subgroup $D$ of $S$ such that any subgroup of $S$ of order $1<|D|<|S|$ and cyclic subgroup of order $4$ (if $S$ is non-abelian, $p=2$ and $|D|=2$) satisfies the $S$-subnormalizer condition in $G$, then we conclude from Remark \ref{Equivalence} that all of them are semi-invariant in $\mathcal{F}_S (G)$. Also, any subgroup of $S$ of order $|D|$ or $p|D|$ is abelian. Hence we obtain from Theorem \ref{Theorem B} that $\mathcal{F}_S (G)$ is supersolvable.
\end{proof}
\begin{proof}[Proof of Corollary \ref{Corollary 2}]
It follows from our hypothesis and Remark \ref{Equivalence} that any maximal subgroup of $S$ is semi-invariant in $\mathcal{F}_S (G)$. Then we conclude immediately from Theorem \ref{C} that $\mathcal{F}_S (G)$ is supersolvable.
\end{proof}
\begin{proof}[Proof of Corollary \ref{Corollary 3}]
It follows directly from Remark \ref{Equivalence} and Theorem \ref{D1}.
\end{proof}
\section{Applications about characterizations for $p$-nilpotency of finite groups }\label{1005}

In this section, we apply the results in Section \ref{1003} and Section \ref{Section 4}, and we obtain 16 corollaries for $p$-nilpotency of finite group $G$ based on 16 different kinds of generalised normalities under the assumption that certain subgroups of $G$ satisfy some generalised normalities conditions.
\begin{corollary}
Let $G$ be a finite group, $p$ be a prime divisor of $|G|$ with $(p-1,|G|)=1$, $S$ be a Sylow $p$-subgroup of $G$. Assume that there exists a subgroup $D$ of $S$ such that $1 <|D|<|S|$, and any subgroup of $S$ of order $|D|$ is abelian and satisfies the $S$-subnormalizer condition in $G$, any subgroup of $S$ of order $p|D|$ is abelian. If $S$ is non-abelian, $p=2$ and $|D|=2$, then assume moreover that any cyclic subgroup of $S$ of order $4$ satisfies the $S$-subnormalizer condition in $G$. Then $G$ is $p$-nilpotent.
\end{corollary}
\begin{proof}
It follows from Corollary  \ref{Corollary 1} that $\mathcal{F}_S (G)$ is supersolvable. By Remark \ref{p-nilpotence}, we conclude that $G$ is $p$-nilpotent and we are done.
\end{proof}

\begin{corollary}
Let $G$ be a finite group, $p$ be an odd prime divisor of $|G|$ with $(p-1,|G|)=1$, $S$ be a Sylow $p$-subgroup of $G$. Suppose that any maximal subgroup of $S$ satisfies the $S$-subnormalizer condition in $G$, and if $S$ is not cyclic, we have additionally that $S$ has more than one abelian maximal subgroup. Then $G$ is $p$-nilpotent.
\end{corollary}
\begin{proof}
It indicates from Corollary  \ref{Corollary 2} that $\mathcal{F}_S (G)$ is supersolvable. Applying Remark \ref{p-nilpotence}, we assert that $G$ is $p$-nilpotent and we are done.
\end{proof}
\begin{corollary}
Let $G$ be a finite group, $p$ be an odd prime divisor of $|G|$ with $(p-1,|G|)=1$, $S$ be a Sylow $p$-subgroup of $G$. Assume that there is a subgroup $D$ of $S$ such that any subgroup of $S$ of order $1 < |D| <|S|$ is abelian and satisfies the $S$-subnormalizer condition in $G$. Then $G$ is $p$-nilpotent.
\end{corollary}
\begin{proof}
It yields from Corollary  \ref{Corollary 3} that $\mathcal{F}_S (G)$ is supersolvable. Then we obtain from Remark \ref{p-nilpotence} that $G$ is $p$-nilpotent, and the proof is complete.
\end{proof}
\begin{remark}
Let $G$ be a finite group and $S$ be a Sylow $p$-subgroup of $G$. Suppose that $H$ is a subgroup of $S$ which is weakly closed in $G$ with respect to $S$. Then for any $g \in N_G (K)$, where $H \leq K \leq S$, $H^g \leq K \leq S$, which implies that $H^g = H$. Hence $H$ satisfies the $S$-subnormalizer condition in $G$. By Lemma \ref{16}, we may change the hypothesis of all theorems above that $H$ satisfies the $S$-subnormalizer condition into $H$ is pronormal in $G$, $H$ is weakly normal in $G$, $H$ is weakly closed in $G$ with respect to $S$ or $N_G (H)$ is the subnormalizer of $H$ in $G$, and we can still obtain that $G$ is $p$-nilpotent.
\end{remark}
\begin{corollary}
Let $G$ be a finite group, $p$ be a prime divisor of $|G|$ with $(p-1,|G|)=1$, $S$ be a Sylow $p$-subgroup of $G$.  Assume that there exists a subgroup $D$ of $S$ with $1<|D|<|S|$ such that every subgroup of $S$ of order $|D|$ or $p|D|$ is abelian, every subgroup of $S$ of order $|D|$ and every cyclic subgroup of $S$ of order 4 (if $S$ is a non-abelian 2-group and $|D|=2$) are generalized $S\Phi$-supplemented in $G$, then $G$ is $p$-nilpotent.
\end{corollary}
\begin{proof}
It indicates from Corollary \ref{generalised} that $\mathcal{F}_S (G)$ is supersolvable. Then we get from Remark \ref{p-nilpotence} that $G$ is $p$-nilpotent.
\end{proof}
\begin{corollary}
Let $G$ be a finite group, $p$ be a prime divisor of $|G|$ with $(p-1,|G|)=1$, $S$ be a Sylow $p$-subgroup of $G$. Suppose that any subgroup $H$ of $S$ with order $p$ or any cyclic subgroup $H$ of $S$ with order $4$ (if $S$ is non-abelian and $p=2$)  is an $ICSC$-subgroup of $G$, then $G$ is $p$-nilpotent.
\end{corollary}
\begin{proof}
It yields from Corollary  \ref{ICSC} that $\mathcal{F}_S (G)$ is supersolvable. Then we obtain from Remark \ref{p-nilpotence} that $G$ is $p$-nilpotent, and the proof is complete.
\end{proof}

\begin{remark}
Using exactly the same proof, we can change the hypothesis that any subgroup $H$ of $S$ with order $p$ or $4$ (if $S$ is non-abelian and $p=2$)  is an $ICSC$-subgroup of $G$ into any cyclic  subgroup $H$ of $S$ with order $p$ or $4$ (if $S$ is non-abelian and $p=2$)  is an $ICSC$-subgroup of $G$. By Lemma \ref{18}, it is easy to find that this theorem actually covers the result of Y. Gao {\it et al} in {{\cite[Theorem 3.2]{GL3}}}.
\end{remark}
\begin{corollary}
Let $G$ be a finite group, $p$ be a prime divisor of $|G|$ with $(p-1,|G|)=1$, $S$ be a Sylow $p$-subgroup of $G$. Suppose that there exists a subgroup $1<D<S$ such that any subgroup of $S$ with order $|D|$ or $p|D|$ is abelian, any subgroup $H$ of $S$ of order $|D|$ or $4$ (If $S$ is a non-abelian $2$-group and $|D| =2$) is a strong $ICC$-subgroup of $G$, then $G$ is $p$-nilpotent.
\end{corollary}
\begin{proof}
It yields from Theorem  \ref{ICC} that $\mathcal{F}_S (G)$ is supersolvable. Then we obtain from Remark \ref{p-nilpotence} that $G$ is $p$-nilpotent, and we are done.
\end{proof}
\begin{corollary}
Let $G$ be a finite group, $p$ be a prime divisor of $|G|$ with $(p-1,|G|)=1$, $S$ be a Sylow $p$-subgroup of $G$.  Suppose that there exists a subgroup $1<D<S$ such that any subgroup of $S$ with order $|D|$ or $p|D|$ is abelian, any subgroup of $S$ of order $|D|$ and any cyclic subgroup of $S$ of order $4$ (If $S$ is a non-abelian $2$-group, $|S|>4$ and $|D| =2$) is strongly $\mathfrak{U}$-embedded in $G$, then $G$ is $p$-nilpotent.
\end{corollary}
\begin{proof}
It indicates from Theorem  \ref{U} that $\mathcal{F}_S (G)$ is supersolvable. Applying Remark \ref{p-nilpotence}, we assert that $G$ is $p$-nilpotent and we are done.
\end{proof}
\begin{corollary}
Let $G$ be a finite group, $p$ be a prime divisor of $|G|$ with $(p-1,|G|)=1$, $S$ be a Sylow $p$-subgroup of $G$. Suppose that any subgroup $H$ of $S$ with order $p$ or any cyclic subgroup $H$ of $S$ with order $4$ (if $S$ is non-abelian and $p=2$)  is $E$-supplemented in $G$, then $G$ is $p$-nilpotent.
\end{corollary}
\begin{proof}
It yields from Theorem  \ref{E-supplemented} that $\mathcal{F}_S (G)$ is supersolvable. By Remark \ref{p-nilpotence}, we conclude that $G$ is $p$-nilpotent and the proof is finished.
\end{proof}

\begin{corollary}
Let $G$ be a finite group, $p$ be a prime divisor of $|G|$ with $(p-1,|G|)=1$, $S$ be a Sylow $p$-subgroup of $G$. Suppose that any subgroup $H$ of $S$ with order $p$ or any cyclic subgroup $H$ of $S$ with order $4$ (if $p=2$)  is a partial $CAP$-subgroup of $G$, then $G$ is $p$-nilpotent.
\end{corollary}
\begin{proof}
It implies from Theorem  \ref{pCAP} that $\mathcal{F}_S (G)$ is supersolvable. Then we obtain from Remark \ref{p-nilpotence} that $G$ is $p$-nilpotent, as needed.
\end{proof}
\begin{corollary}
Let $G$ be a finite group, $p$ be an odd prime divisor of $|G|$ with $(p-1,|G|)=1$, $S$ be a Sylow $p$-subgroup of $G$. Suppose that $S$ has a subgroup $D$ such that $1 \leq D <S$, and all subgroups of $S$ of order $|D|$ or $p|D|$ are abelian and $c$-supplemented in $G$, then $G$ is $p$-nilpotent.
\end{corollary}
\begin{proof}
It indicates from Theorem  \ref{c-supplemented} that $\mathcal{F}_S (G)$ is supersolvable. Then we conclude from Remark \ref{p-nilpotence} that $G$ is $p$-nilpotent, as follows.
\end{proof}
\begin{corollary}
Let $G$ be a finite group, $p$ be a prime divisor of $|G|$ with $(p-1,|G|)=1$, $S$ be a Sylow $p$-subgroup of $G$. Suppose that one of the following holds:
\begin{itemize}
\item[(1)] Every cyclic subgroup of $S$ of order $p$ or $4$ (if $P$ is a non-cyclic $2$-group) either is a weakly $\mathcal{H}$-subgroup of $G$ or has a supersoluble supplement in $G$.
\item[(2)] $p$ is odd, and there exists a subgroup $D$ of $S$ such that every subgroup of $S$ of order $|D|$ is a weakly $\mathcal{H}$-subgroup of $G$, every subgroup of $S$ of order $|D|$ or $p|D|$ is abelian.
\end{itemize} 
Then $G$ is $p$-nilpotent.
\end{corollary}
\begin{proof}
It yields from Theorem  \ref{weakly-H-1} and  \ref{weakly-H-2} that $\mathcal{F}_S (G)$ is supersolvable. Then we assert from Remark \ref{p-nilpotence} that $G$ is $p$-nilpotent and we are done.
\end{proof}
\begin{corollary}
Let $G$ be a finite group, $p$ be a prime divisor of $|G|$ with $(p-1,|G|)=1$, $S$ be a Sylow $p$-subgroup of $G$.  Suppose that $S$ has a subgroup $D$ such that $1 \leq D <S$, and all subgroups of $S$ of order $|D|$ or $p|D|$ are abelian and weakly $S\Phi$-supplemented in $G$. If $S$ is a non-abelian $2$-group and $|D|=1$, we have additionally that every cyclic subgroup of $G$ of order $4$ is a weakly $S\Phi$-supplemented subgroup of $G$.  Then  $G$ is $p$-nilpotent.
\end{corollary}
\begin{proof}
It indicates from Theorem  \ref{weakly SPhi} that $\mathcal{F}_S (G)$ is supersolvable. Then we conclude from Remark \ref{p-nilpotence} that $G$ is $p$-nilpotent.
\end{proof}
\begin{corollary}
Let $G$ be a finite group, $p$ be a prime divisor of $|G|$ with $(p-1,|G|)=1$, $S$ be a Sylow $p$-subgroup of $G$. Suppose that $S$ has a subgroup $D$ such that $1 < |D |<|S|$, and all subgroups of $S$ of order $|D|$ or $p|D|$ are abelian, every subgroup of order $|D|$ or $4$ (if $p = |D|=2$) is an $\mathscr{H} C$-subgroup of $G$.  Then  $G$ is $p$-nilpotent.
\end{corollary}
\begin{proof}
It follows from Theorem  \ref{HC} that $\mathcal{F}_S (G)$ is supersolvable. Then we obtain from Remark \ref{p-nilpotence} that $G$ is $p$-nilpotent.
\end{proof}
\begin{corollary}
Let $G$ be a finite group, $p$ be a prime divisor of $|G|$ with $(p-1,|G|)=1$, $S$ be a Sylow $p$-subgroup of $G$. Suppose that there exists a subgroup $D$ of $S$ such that $1<|D|<|S|$, every subgroup of $S$ of order $|D|$ or $p|D|$ is abelian, every subgroup of $S$ of order $|D|$ or $2|D|$ (if $S$ is non-abelian, $p=2$ and $|S:D|>2$) is either $S$-quasinormally embedded or $SS$-quasinormal in $G$, then $G$ is $p$-nilpotent.
\end{corollary}
\begin{proof}
It indicates from Theorem  \ref{SS} that $\mathcal{F}_S (G)$ is supersolvable. Then we assert from Remark \ref{p-nilpotence} that $G$ is $p$-nilpotent and the proof is complete.
\end{proof}
\begin{corollary}
Let $G$ be a finite group, $p$ be a prime divisor of $|G|$ with $(p-1,|G|)=1$, $S$ be a Sylow $p$-subgroup of $G$. Suppose that there exists a subgroup $D$ of $S$ such that $1<|D|<|S|$, every subgroup of $S$ of order $|D|$ or $p|D|$ is abelian and weakly $\mathcal{HC}$-embedded in $G$, then $G$ is $p$-nilpotent.
\end{corollary}
\begin{proof}
It yields from Theorem  \ref{weakly HC-embedded} that $\mathcal{F}_S (G)$ is supersolvable. Then we obtain from Remark \ref{p-nilpotence} that $G$ is $p$-nilpotent and the proof is finished.
\end{proof}

\begin{corollary}
Let $G$ be a $Q_8$-$free$ finite group, $p$ be a prime divisor of $|G|$ with $(p-1,|G|)=1$, $S$ be a Sylow $p$-subgroup of $G$. Suppose that any subgroup $H$ of $S$ with order $p$ is an $NE^{*}$-subgroup of $G$, then $G$ is $p$-nilpotent.
\end{corollary}
\begin{proof}
It implies from Corollary \ref{NE*} that $\mathcal{F}_S (G)$ is supersolvable. Then we conclude from Remark \ref{p-nilpotence} that $G$ is $p$-nilpotent and we are done.
\end{proof}

\end{document}